\documentclass[12pt]{extarticle}
\usepackage{amsmath, amsthm, amssymb, mathtools, hyperref,color}
\usepackage{graphicx}
\usepackage{caption}
\usepackage{subcaption}
\usepackage{tikz}
\usetikzlibrary{calc}
\tikzstyle{box} = [rectangle, minimum width=3cm, minimum height=1cm, text width=6cm, text centered, draw=black]
\tikzstyle{arrow} = [thick,->,>=stealth]

\newtheorem{theorem}{Theorem}
\numberwithin{theorem}{section}
\newtheorem{proposition}[theorem]{Proposition}
\newtheorem{lemma}[theorem]{Lemma}
\newtheorem{corollary}[theorem]{Corollary}
\newtheorem{definition}[theorem]{Definition}
\newtheorem{remark}[theorem]{Remark}
\newtheorem{conjecture}[theorem]{Conjecture}

\newcommand{\legendre}[2][p]{\ensuremath{\left( \frac{#2}{#1} \right) }}
\title{Asymptotic Brill-Noether Existence at the Half-Canonical Degree \\
\large{Energy Pairing, Cheeger Inequality and Covering Radii}}

\author{Madhusudan Manjunath\footnote{Part of this work was carried out while visiting IHES, Bures-sur-Yvette. We thank IHES for its generous support and warm hospitality.} }

\begin{document}
\maketitle
\begin{abstract}
We study \emph{asymptotic versions} of the Brill-Noether existence conjecture on graphs via techniques inspired by the \emph{geometry of numbers}. We confirm an asymptotic version of the conjecture at (and near) the \emph{half-canonical degree} in several well-connected families of graphs. They include expander graphs of even valence, almost-Ramanujan graphs of a fixed valence at least five and certain random graphs. In particular, for any fixed $k \geq 5$, almost all simple, connected, $k$-regular graphs satisfy the Brill-Noether existence conjecture at the half-canonical degree up to a constant factor. The key tool is a \emph{Cheeger-style inequality} for the covering radius of a certain periodic set with respect to the energy quadratic form associated with the graph. As an application, we lower bound the diameter of graphs associated with certain dynamical systems called \emph{reversal systems}. We conclude with a suggestion to tackle the asymptotic version of the conjecture, in general, i.e. beyond half-canonical degrees. 
\end{abstract}




\section{Introduction}

Brill-Noether theory, in its most basic form, concerns the existence of divisor classes of prescribed degree and rank on an ``appropriate'' object.  Typical objects include compact Riemann surfaces, projective algebraic curves, tropical curves and finite graphs. 

We begin by recalling classical Brill-Noether theory. Let $X$ be a smooth projective curve of genus $g$ over an algebraically closed field.
Define the \emph{Brill-Noether number} $\rho(g,r,d):=g-(r+1)(g-d+r)$.   The Brill-Noether theorem \cite[Chapter 5]{ArbCorGriHar85} consists of two parts: the \emph{existence} part and the \emph{non-existence} part. Let $d_0, r_0$ be non-negative integers.
The existence part asserts that if $\rho(g,r_0,d_0) \geq 0$, then $X$ has a divisor of degree $d_0$ and rank at least $r_0$ (or equivalently, a divisor of degree at most $d_0$ and rank equal to $r_0$).  The non-existence part asserts that for a general $X$,
if  $\rho(g,r_0,d_0)<0$, then $X$ does not have a divisor of degree $d_0$ and rank at least $r_0$.

More recently, there has been considerable interest in Brill-Noether theory of  i. \emph{special curves}, in particular curves general in their gonality stratum \cite{JenRan21, Vog26}, 
ii. \emph{vector bundles} on curves \cite{Tei91,Muk96}, iii. \emph{finite graphs} and their metric variants, i.e. \emph{abstract tropical curves} \cite{Bak08}.  This article is dedicated to the third direction, in particular \emph{Brill-Noether theory on a finite graph}.

In 2008, Baker in his seminal work \cite{Bak08} formulated a Brill-Noether theory on finite and metric graphs. In particular, he laid out conjectures analogous to the existence and non-existence theorems. Baker's \emph{Brill-Noether existence conjecture for finite graphs} \cite[Item (1), Conjecture 3.9]{Bak08}\label{BNex_conj} is the focus of this article and is as follows. 

\begin{conjecture}{{\rm \bf{(Brill-Noether Existence Conjecture for Graphs)}}}
Let $G$ be a finite, undirected, connected, loop-free, multigraph of genus $g$.  For every pair of non-negative integers $r_0,d_0$ satisfying $\rho(g,r_0,d_0) \geq 0$, there is a divisor $D$ on $G$ of degree $d_0$ and rank at least $r_0$ (or equivalently, a divisor of degree at most $d_0$ and rank equal to $r_0$).
\end{conjecture}

Baker proved Brill-Noether existence on \emph{metric graphs} \cite[Theorem 3.20]{Bak08} using its classical counterpart and his specialisation lemma \cite[Lemma 2.8]{Bak08}. Baker \cite[Theorem 3.12]{Bak08} also shows that if arbitrary edge subdivisions are allowed then a corresponding ``weak version" of Brill-Noether existence for finite graphs follows. However, no such technique is known for the conjecture in its original form (Conjecture \ref{BNex_conj}). As Baker points out \cite[Remark 3.13]{Bak08}, obtaining ``combinatorial" proofs, not relying on algebraic curve theory, of Brill-Noether existence even in the case of metric graphs is of high interest. Draisma and Vargas \cite{DraVar19} obtain a combinatorial proof in the case of metric graphs for  (a variant of) $r=1$. We refer to \cite{CooDraPayRob12} for the non-existence part.

The Brill-Noether Existence Conjecture for Graphs (Conjecture \ref{BNex_conj}) is wide open.\footnote{Caporaso \cite{Cap11} claimed a proof of Conjecture \ref{BNex_conj} using algebro-geometric techniques. However, this proof is known to contain a gap \cite{BakJen16}.} 
It has been confirmed for specific families of graphs such as complete graphs \cite{CorLeb16, ABNCDJN26},  cactus graphs \cite{Duo22} and ``low'' genus graphs (up to five)  \cite{AtaRan18}.  There have also been announcements confirming it for wheel and fish graphs \cite{CorDuTra25,BakLeaSmiSol26}. 
The chief difficulty in tacking Brill-Noether existence is the lack of general techniques to lower bound the rank of divisors (of a fixed degree $d_0$) in terms of  $d_0$ and appropriate graph invariants. To the best of our knowledge, all the current techniques rely on \emph{ the ``simplicity" of the Picard group of the graph }and \emph{the theory of reduced divisors} to effectively determine the ranks of all (or sufficiently many) divisor classes. In general, both the Picard group and the theory of reduced divisors seem ``complicated'' from this viewpoint and hence, obstructing progress. 

In this article, we introduce an approach (currently, most effective near the half-canonical degree) to Brill-Noether existence on arbitrary regular graphs via \emph{certain quadratic forms on the Laplacian lattice} $L_G$, i.e. the lattice generated by the rows of the Laplacian matrix of $G$.  This technique is inspired by \emph{a ``geometry of numbers''} \cite{Cas71} approach to Brill-Noether existence for graphs and builds on \cite{Man-crc22}. The author in \cite{Man-crc22} reformulated (the $\mathbb{R}$-divisor version of) Brill-Noether existence as the \emph{covering radius conjecture (CRC}), Conjecture \ref{crc_conj}. The CRC asserts lower bounds on the covering radii of a \emph{periodic set} ${\rm Crit}_{\triangle}(L_G)$ with respect to certain \emph{generalised permutahedra} \cite{Pos09}. This articles studies the CRC in terms of a certain quadratic form on $L_G$ called \emph{the energy quadratic form}. The bilinear pairing underlying this form is called the \emph{energy pairing}  \cite{BakSho13}.
At the heart of our approach is a \emph{Cheeger-like inequality} (Lemma \ref{coveglg}) for the covering radius of ${\rm Crit}_{\triangle}(L_G)$ with respect to the energy quadratic form.  This inequality is in terms of the spectral gap of $G$ and per se, does not involve the Picard group or reduced divisors.  Via this Cheeger inequality, we show \emph{asymptotic versions of Brill-Noether existence near the half-canonical degree} in various natural families of graphs.  These families are: 

\begin{itemize}

\item Spectral expander graphs of fixed even valence. 

\item Almost-Ramanujan graphs of fixed valence (at least five).

\item Almost-Ramanujan graphs of unbounded valence.

\item Regular graphs of a fixed valence at least five, drawn according to a random model. Two random models we treat are the types $\mathcal{U}$ and $\mathcal{I}$. Type $\mathcal{U}$ consists of uniform distributions $\mathcal{U}_{n,k}$ on simple, $k$-regular, connected graphs on $n$ vertices.  Type $\mathcal{I}$ is defined in terms of perfect matchings, we refer to Subsection \ref{expramran_subsect} for more details.

\end{itemize}

{\bf Notation and Terminology:}  Let $I \subseteq \mathbb{N}$ be an infinite set.  This set will index families of objects: typical objects are graphs, divisors and probability distributions. In the case of a family of graphs $\{G_n\}_{n \in I}$ indexed by $I$, we shall assume that $|V(G_n)|=n$. By the term ``graph", we mean a finite, undirected, connected, loop-free graph with possibly multiple edges unless stated otherwise. The genus $g$ of a graph $G$ is its first Betti number a.k.a. cyclomatic number given by $|E(G)|-|V(G)|+1$.

 In the half-canonical degree, the Brill-Noether existence conjecture asserts the existence of a divisor of rank at least $\lfloor \sqrt{g} \rfloor-1$.  This leads to the following asymptotic notions. We say that a family $\mathcal{F}=\{G_n\}_{n \in I}$ of graphs satisfies \emph{asymptotic Brill-Noether existence (ABNE) at half-canonical degree up to a factor $C(n)$} if there exists a sequence $\{D_n\}_{n \in I}$ of divisors $D_n$ on $G_n$ such that the degree of $D_n$ is $g(n)-1$ and rank $r(D_n) \in \Omega{(C(n) \sqrt{g(n)}})$ where $g(n)$ is the genus of $G_n$. In cases where $C(n)$ can be chosen to be a constant, we do not mention the factor explicitly.

A probabilistic version of this notion is as follows.  Let $\{\mathfrak{D}_n\}_{n \in I}$ be a sequence where each $\mathfrak{D}_n$ is a probability distribution on a finite set of regular graphs of valence at least three.  Let $\mathcal{F}=\{G_n\}_{n \in I}$ be a sequence of graphs where each $G_n$ is drawn according to $\mathfrak{D}_n$.   For a real $C>0$,  let $p_C(n)$ be the probability that there exists a divisor $D_n$ on $G_n$ of half-canonical degree and  rank $r(D_n) \geq C \cdot \sqrt{g(n)}$ (conditioned on $G_n$ being connected). We say that $\mathcal{F}$ satisfies \emph{ABNE at half-canonical degree with high probability} if there exists a constant $C>0$ such that $p_C(n) \rightarrow 1$ as $n \rightarrow \infty$.







\begin{theorem}\label{ramranintro_theo}
 Let $\mathcal{F}=\{G_n\}_{n \in I}$ be a family of regular graphs such that $|V(G_n)|=n$ for all $n \in I$ and let $k(n) \geq 3$ be the valence of $G_n$.  The family $\mathcal{F}$ satisfies ABNE at half-canonical degree  
  in the following cases:
 
 \begin{enumerate}
 
 \item \label{exp_cas} If $k(n)=k$ is an even constant (independent of $n$) and  $\mathcal{F}$ is a family of $k$-regular spectral expander graphs.
 
\item\label{ramcon_cas} If $k(n)=k \geq 5$ is a constant and $\mathcal{F}$ is a family of $k$-regular  almost-Ramanujan graphs.
 
 \item \label{ran_cas}  With high probability, if $k(n)=k \geq 5$ is a constant and $\mathcal{F}$  is a family of $k$-regular random graphs of either Type $\mathcal{U}$ or Type $\mathcal{I}$.
 
 \item \label{ramunb_cas}  Up to a factor $1/\sqrt{\log_{k(n)}n}$, if it is a family of almost-Ramanujan graphs of unbounded valence $k(n) \rightarrow \infty$.


 \end{enumerate} 
\end{theorem}

Theorem \ref{ramranintro_theo} confirms the predictions of the Brill-Noether existence conjecture, up to a constant, in Cases (\ref{exp_cas}), (\ref{ramcon_cas}) and (\ref{ran_cas}), and up to $C/\sqrt{\log_{k(n)}n}$, for an absolute constant $C$, in Case (\ref{ramunb_cas}).
An informal version of Case (\ref{ran_cas}) for Type $\mathcal{U}$ is as follows: for any fixed $k \geq 5$, almost all simple, $k$-regular, connected graphs satisfy Brill-Noether existence at half-canonical degree up to a constant factor (possibly depending on $k$ but not on $n$).

{\bf Significance of the Half-Canonical Degree:}  The half-canonical degree houses several important divisor classes.  Namely, theta characteristics and the locus $W_{g-1}$ of special divisor classes of degree $g-1$.
These topics are rooted in the classical theory of theta functions \cite[Chapter I]{Mum07}. They have also been studied more recently in the context of algebraic curves \cite[Chapter II]{Mum07} and tropical curves \cite{MikZha08}. Recall that a divisor class $[D]$ (on either a smooth algebraic curve, a tropical curve or a graph) is called a \emph{theta characteristic} if $[2D]$ is the canonical class. One difficultly in handling the half-canonical degree, from the viewpoint of Brill-Noether theory, is that the Riemann-Roch formula yields relatively little information about the rank. For instance, Riemann-Roch yields no information at all about the rank of a theta characteristic. In general, for a divisor $D$ of half-canonical degree, it only says $r(D)=r(K-D)$ where $K$ is an element in the canonical divisor class.  Furthermore, we believe that the half-canonical degree case serves as a ``proof of concept" offering a framework for the Brill-Noether existence conjecture in its full generality.  We refer to Section \ref{bey_sect} for more on this.



 {







\subsection{Proof Technique}\label{prooftech_sect}

Our proof is geometric and is based on the \emph{covering radius conjecture (CRC)}. At the half-canonical degree, the CRC asserts a lower bound of $\sqrt{g}$ on the covering radius of ${\rm Crit}_{\triangle}(L_G)$ with respect to the $\ell_1$-norm.\footnote{equivalently, in its original formulation, the CRC asserts a lower bound of $\frac{\sqrt{g}}{n}$ on the covering radius of ${\rm Crit}_{\triangle}(L_G)$ with respect to the Minkowski sum $\triangle+\bar{\triangle}$ of opposite simplices, see Subsection \ref{laplatdisfun_subsect} for more details.} A direct approach to the CRC is to compute the Voronoi diagram of ${\rm Crit}_{\triangle}(L_G)$ with respect to the  $\ell_1$-norm. However, this currently seems out of reach. The article \cite{Man-crc22} studies the CRC via simplicial distance functions. As remarked in \cite[Page 7]{Man-crc22}, this approach proves the CRC for dense graphs ($g$ quadratic in the number of vertices), it is ineffective for sparse graphs such as regular graphs with a fixed valence.

Instead, we use the energy quadratic form $\mathcal{E}_G$ as a ``proxy" for the $\ell_1$-norm. The involvement of a quadratic form is not apparent in either the Brill-Noether existence conjecture or the CRC. How to sense it then? The \emph{square root} in the CRC suggests this hidden hand of quadratic forms: ``how else can a square root come rather than as a covering radius with respect to a quadratic form?". The next natural question is: Why the \emph{energy} quadratic form? If $G$ is regular, then this quadratic form has the desirable property that the set of holes of $L_G$ with respect to it is precisely ${\rm Crit}_{\triangle}(L_G)$. Determining the Voronoi diagram of ${\rm Crit}_{\triangle}(L_G)$ or even its covering radius with respect to $\mathcal{E}_G$ also seems out of reach. Instead, we derive a \emph{Cheeger-style lower bound} (Lemma \ref{coveglg}) for this covering radius. 

We then employ norm conversion inequalities to derive lower bounds on the covering radius of ${\rm Crit}_{\triangle}(L_G)$ with respect to the $\ell_1$-norm (Section \ref{bnhc_sect}). We use two types of inequalities: i spectral-based for the constant valence case, ii. Poisson equation-based for unbounded valence. These methods, combined with previously known results on expander and Ramanujan graphs (see Subsection \ref{expramran_subsect}), show that the sequence of half-canonical divisors satisfies Theorem \ref{ramranintro_theo}.  
However, in general, the half-canonical divisor is only a $\mathbb{Q}$-divisor. 

 For a $k$-regular graph, it is a ($\mathbb{Z}$)-divisor if and only if $k$ is even. In the odd case, we add another $\mathbb{Q}$-divisor that is ``close enough'' to the half-canonical divisor with respect to the (metric induced by the) energy quadratic form. The result is a sequence of $\mathbb{Z}$-divisors that satisfies Theorem \ref{ramranintro_theo}. These divisors are constructed from vertex subsets of size $|V(G)|/2$, see Subsection \ref{qtoz_sect} for more details. 
 
 Figure \ref{prftec_fig} depicts the main steps in the proof of Theorem  \ref{ramranintro_theo}.

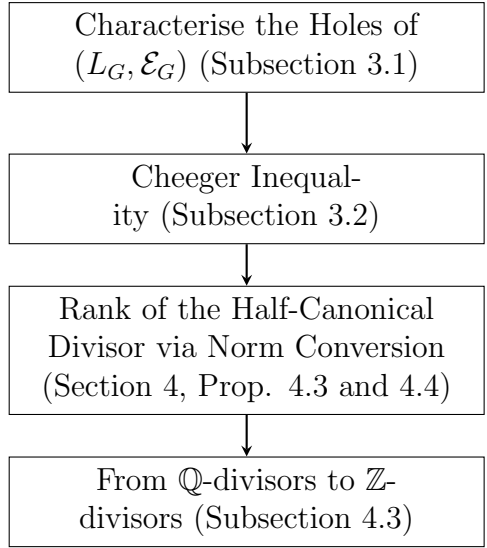
\begin{figure}
\begin{center}
\begin{tikzpicture}[node distance=2cm]
\node (holeener_box) [box] {Characterise the Holes of $(L_G, \mathcal{E}_G)$ (Subsection \ref{lapenehol_subsect})};
\node (che_box) [box, below of=holeener_box] {Cheeger Inequality (Subsection \ref{che_subsect})};
\node (lbr_box) [box, below of=che_box] {Rank of the Half-Canonical Divisor via Norm Conversion (Section \ref{bnhc_sect}, Prop. \ref{halfcanrank_prop} and \ref{halfcanrankdia_prop})};
\node (qtoz_box) [box, below of=lbr_box] {From $\mathbb{Q}$-divisors to $\mathbb{Z}$-divisors (Subsection \ref{qtoz_sect})};
\draw [arrow] (holeener_box) -- (che_box);
\draw [arrow] (che_box) -- (lbr_box);
\draw [arrow] (lbr_box) -- (qtoz_box);
 \end{tikzpicture}
\caption{Main Steps of the Proof of Theorem  \ref{ramranintro_theo}}\label{prftec_fig}
\end{center}
 \end{figure}

\subsection{Applications of Theorem \ref{ramranintro_theo}}

We present the following applications of Theorem \ref{ramranintro_theo}.

{\bf BNE at Nearly Half-Canonical Degrees:}
Theorem \ref{ramranintro_theo} can also handle degrees that are ``nearly" half-canonical. 

\begin{corollary}\label{nhc_cor}
Let $d_0: I \rightarrow \mathbb{N}$, write $d_0(n)=g(n)-1-\psi(n)$ for all $n \in I$. The following cases hold:
\begin{itemize}
\item  If $\psi(n) \in o(\sqrt{g(n)})$ and the family $\mathcal{F}$ satisfies either Item (\ref{exp_cas}), (\ref{ramcon_cas}) or (\ref{ran_cas}), Theorem \ref{ramranintro_theo}, then there exists a sequence $\{D_n\}_{n \in I}$ where $D_n$ is a divisor on $G_n$ with degree $d_0(n)$ and $r(D_n) \in \Omega(\sqrt{g(n)})$.
 \item If  $\psi(n) \in o(\sqrt{g(n)/\log_{k(n)}n})$and $\mathcal{F}$ satisfies Item (\ref{ramunb_cas}), Theorem \ref{ramranintro_theo}, then there exists such a sequence $\{D_n\}_{n \in I}$ with $d_0(n)$ as the degree of $D_n$ and $r(D_n) \in \Omega(\sqrt{g(n)/\log_{k(n)}n})$.
 \end{itemize}
\end{corollary}

The sequence $\{D_n\}_{n \in I}$ in Corollary \ref{nhc_cor} is constructed by subtracting effective divisors $E_n$ of degree $\psi(n)$ from the corresponding sequence obtained from Theorem \ref{ramranintro_theo} and noting that the rank cannot drop by more than $\psi(n)$.  Corollary \ref{nhc_cor} matches the predictions of the Brill-Noether existence conjecture, up to a constant, in Items (\ref{exp_cas}), (\ref{ramcon_cas})  and (\ref{ran_cas}), and up to $C/\sqrt{\log_{k(n)}n}$ in Item (\ref{ramunb_cas}).

{\bf Reversal Systems on Graphs:} Theorem \ref{ramranintro_theo} can be used to lower bound the diameter of certain graphs arising from these dynamical systems. This topic is treated in Section \ref{revsys_sect}.

{\bf Acknowledgements:} We thank the participants of the workshop ``Combinatorial and Algebraic Aspects on Lattice Polytopes'', Kobe, Japan, February 2023 for their interest in this topic. We thank Martin Ulirsch for the valuable discussions in March 2025.

\section{Preliminaries}

In this section, we introduce the key players of this article and set up their notation. We shall use this freely in the rest of the article. Let $G$ be an undirected, connected (multi) graph without loops on $n$ vertices, with $m$ edges and of genus $g$. 

\subsection{Oriented Edges, Cochains and Laplacians}
In the following, we adapt notation from \cite{BacLarNag97}. Let $\mathbb{E}(G)$ be the set of ordered pairs $(u,v)$ and $(v,u)$ for each edge $\{u,v\}$ of $G$. 
The elements of $\mathbb{E}(G)$ are called \emph{oriented edges}. For each oriented edge $e=(u,v) \in \mathbb{E}(G)$, we refer to $(v,u)$ as the \emph{inverse} of $e$ and denote it by $\bar{e}$. 

For a commutative ring $A$ with unity, let the \emph{$0$-cochains} $C^0(G,A)$ be the free $A$-module of all functions $f:V(G) \rightarrow A$ and let the \emph{$1$-cochains} $C^1(G,A)$ be the free $A$-module of all functions $h: \mathbb{E}(G) \rightarrow A$ such that $h(\bar{e})=-h(e)$ for all oriented edges.  For both $C^0(G,A)$ and  $C^1(G,A)$, the indicator functions at the vertices and the oriented edges~\footnote{an indicator function at $e$ takes value one at $e$, minus one at $\bar{e}$ and zero at all other oriented edges.}: one for each edge, respectively form a basis, and are orthonormal with respect to the inner products $\langle f_1, f_2 \rangle_v=\sum_{v \in V(G)}f_1(v)\cdot f_2(v)$ and $\langle h_1,h_2 \rangle_e=(\sum_{e \in \mathbb{E}(G)} h_1(e) \cdot h_2(e))/2$.  

The \emph{exterior differential} $\delta: C^0(G,A) \rightarrow C^{1}(G,A)$ is given by  $(\delta(f))(e)=f(e_+)-f(e_-)$ where $e_+$ and $e_-$ are the head and tail vertices of the oriented edge $e$, respectively. Note that $(\delta(f))(\bar{e})=-(\delta(f))(e)$ for every oriented edge $e$ and this makes $\delta$ well-defined. Its adjoint $\delta^*: C^1(G,A) \rightarrow C^0(G,A)$ is given by $(\delta^*(h))(v)=\sum_{e \in \mathbb{E}(G)|~e_+=v} h(e)$. 

\begin{definition}{\rm {\bf(Laplacian Operator)}} 
The Laplacian operator is the composite $\delta^* \circ \delta: C^0(G,A) \rightarrow C^0(G,A)$ and is denoted by $\Delta$.
\end{definition}
 Recall the formula $(\Delta(f))(v)=\sum_{(v,u) \in \mathbb{E}(G)} (f(v)-f(u))={\rm val}(v) \cdot f(v)- \sum_{(v,u) \in \mathbb{E}(G)} f(u)$ where ${\rm val}(.)$ is the valence.  

\subsection{The Laplacian Lattice and Distance Functions}\label{laplatdisfun_subsect}  Recall that the \emph{Laplacian matrix} is the matrix of the Laplacian operator with $A=\mathbb{R}$ and with respect to the standard basis on $C^0(G,\mathbb{R})$.

\begin{definition}{\rm {\bf(Laplacian Lattice)}} 
The Laplacian lattice $L_G$ of $G$ is the lattice generated by the rows (or equivalently, the columns) of the Laplacian matrix of $G$.
\end{definition}

Since the row sum of the Laplacian matrix is zero, the Laplacian lattice $L_G$ is contained in the hyperplane $H_0=(1,\dots,1)^{\perp}$. Furthermore, it is a finite index sublattice of  the root lattice $A_{n-1}$ and $[A_{n-1}:L_G]=c_G$ where 
$c_G$ is the number of spanning trees of $G$.

A geometric approach to Riemann-Roch and Brill-Noether theories in the graph theoretic setting involves Laplacian latttices equipped with suitable distance functions \cite{AmiMan10,Man-sim13, Man-crc22}. Let $\mathfrak{C}$ be a convex body (compact, convex set of full dimension) in $H_0$ and containing the origin in its interior. Let $H_d$ be the hyperplane \begin{center}$\{(p_1,\dots,p_n)|~\sum_{i=1}^{n}p_i=d\}$. \end{center}
Fix a real number $d$, the distance function $d_{\mathfrak{C}}:H_d \times H_d \rightarrow \mathbb{R}_{\geq 0}$ induced by $\mathfrak{C}$ is
\begin{center}
$d_{\mathfrak{C}}({\bf p_1},{\bf p_2})={\rm inf} \{ \mu|~ {\bf p_2} \in {\bf p_1}+ \mu \cdot \mathfrak{C}\}$.
\end{center}
The function $d_{\mathfrak{C}}$ satisfies all properties of a metric except possibly symmetry, i.e. $d_{\mathfrak{C}}({\bf p_1},{\bf p_2})=d_{\mathfrak{C}}({\bf p_2},{\bf p_1})$.  

Chief examples of such convex bodies featuring in this article are the following:

\begin{enumerate}
\item The standard simplices $\triangle$ and $\bar{\triangle}$. They are defined as $\triangle={\rm CH}({\bf b_1},\dots,{\bf b_n})$ where ${\rm CH}(.)$ is the convex hull, each ${\bf b_i}=(n-1) \cdot {\bf e_i}-\sum_{j \neq i} {\bf e_j}$ and ${\bf e_k}$ is the $k$-th standard basis vector. The standard simplex $\bar{\triangle}$ is defined as $-\triangle$.

\item Their scaled Minkowski sum $P_{\alpha,\bar{\alpha}}:=\alpha \cdot \triangle + \bar{\alpha} \cdot \bar{\triangle}$ for non-negative reals $\alpha,\bar{\alpha}$. These are instances of generalised permutahedra.

\item Ellipsoid given as the unit ball of a quadratic form. 
\end{enumerate}

Let $\mathcal{L}$ be a lattice in $H_0$ of full rank. The lattice $\mathcal{L}$ acts on $H_d$ via translation. Let $T$ be a subset of $H_d$, for some real $d$, that is discrete with respect to the Euclidean topology and  inheriting the translation action of  $\mathcal{L}$. The distance function $h_{\mathfrak{C},T}: H_d \rightarrow \mathbb{R}_{\geq 0}$ on $H_d$ with respect to $\mathfrak{C}$ and $T$ is defined as
\begin{center}
$h_{\mathfrak{C},T}({\bf p})={\rm min}_{{\bf r} \in T} d_{\mathfrak{C}}({\bf p},{\bf r})$.
\end{center}

The discreteness of $T$ ensures that $h_{\mathfrak{C},T}$ is well-defined.  

\begin{lemma}\label{pwis-triine_lem}
Let $d_{\mathfrak{C}_1}$ and $d_{\mathfrak{C}_2}$ be distance functions and let ${\bf p} \in H_d$.  The following properties hold: 
\begin{enumerate}
\item\label{pwis_it} If $d_{\mathfrak{C}_1}({\bf p}, {\bf r}) \geq d_{\mathfrak{C}_2}({\bf p}, {\bf r})$ for every ${\bf r} \in T$, then $h_{\mathfrak{C}_1,T}({\bf p}) \geq h_{\mathfrak{C}_2,T}({\bf p}) $. 

\item\label{triine_it}  {\rm {\bf (Triangle Inequality)}} For any ${\bf p},{\bf q}  \in H_d$, $h_{\mathfrak{C},T}({\bf p}) \leq h_{\mathfrak{C},T}({\bf q})+d_{\mathfrak{C}}({\bf p}, {\bf q})$.

\end{enumerate}
\end{lemma}

The function $h_{\mathfrak{C},T}$ is an $\mathcal{L}$-periodic, continuous function. Hence, it induces a continuous function on the topological torus $H_d/\sim_\mathcal{L}$. Thus, it is bounded and attains its infimum and supremum. 
Its minimum is zero and is attained on the set $T$. Its maximum is called the \emph{covering radius} of $T$ with respect to $\mathfrak{C}$.

\begin{definition}{\rm {\bf(Covering Radius)}} 
The maximum value of the function $h_{\mathfrak{C},T}: H_d \rightarrow \mathbb{R}_{\geq 0}$ is called the covering radius of $T$ with respect to $\mathfrak{C}$ and will be denoted by ${\rm Cov}_{\mathfrak{C}}(T)$.
\end{definition}
In cases where $\mathfrak{C}$ is given implicitly as a unit ball of either a metric $||.||$ or a quadratic form $Q$, we also use the notations ${\rm Cov}_{||.||}(T),~h_{||.||,T}$ and ${\rm Cov}_{Q}(T),~h_{Q,T}$, respectively.
The set of \emph{local maxima} of $h_{\mathfrak{C},T}$ is denoted by ${\rm Crit}_{\mathfrak{C}}(T)$. The elements in ${\rm Crit}_{\mathfrak{C}}(T)$ are called the \emph{holes of $(\mathfrak{C},T)$}. The case when $\mathfrak{C}$ is a standard simplex and $T=L_G$ plays a central role in this article. In particular, lower bounding the covering radius of ${\rm Crit}_{\triangle}(L_G)$ with respect to $P_{\alpha,\bar{\alpha}}$ (for certain $\alpha,\bar{\alpha}$) is equivalent to the $\mathbb{R}$-divisor version of Brill-Noether existence. This lower bound is the \emph{covering radius conjecture} (CRC).

\begin{conjecture}{\rm {\bf (CRC)}}\label{crc_conj}
Let $G$ be a graph of genus $g \geq 1$.  For positive reals $\alpha,\bar{\alpha}$ such that $\alpha/{\bar{\alpha}} \in [1/g,g]$, the covering radius ${\rm Cov}_{P_{\alpha,\bar{\alpha}}}({\rm Crit}_{\triangle}(L_G)) \geq \dfrac{\sqrt{g}}{n \cdot \sqrt{\alpha \cdot \bar{\alpha}}}$.
\end{conjecture}

\begin{remark} {\rm {\bf (Symmetries of the CRC)}} 
\rm
 The predicted lower bound remains invariant under  $\alpha \leftrightarrow \bar{\alpha}$. By the \emph{reflection invariance} of $L_G$ \cite[Definition 5.1]{AmiMan10}, ${\rm Crit}_{\bar{\triangle}}(L_G)$ is a translate of ${\rm Crit}_{\triangle}(L_G)$ and hence, their covering radii, with respect to any fixed distance function, are equal. \qed
\end{remark}


By \cite[Item (ii), Theorem 6.9]{AmiMan10}, the set ${\rm Crit}_{\triangle}(L_G)$ has the following combinatorial description.  Let $\mathcal{N}_G$ be the set of \emph{non-special} divisors on $G$, i.e. divisors of degree $g-1$ and rank minus one.   In the following, we regard $\mathcal{N}_G$
as a subset of $H_{g-1}$ and for $d \in \mathbb{R}$, let $\pi_d$ be the projection map $\pi_d({\bf p})={\bf p}-(\frac{\sum_i p_i-d}{n})\cdot (1,\dots,1)$ where ${\bf p}=(p_1,\dots,p_n)$.

\begin{theorem}\label{critnonspe_theo}\cite[Item (ii), Theorem 6.9]{AmiMan10} 
The set ${\rm Crit}_{\triangle}(L_G)$ is equal to $-\pi_0(\mathcal{N}_G)$.   If $G$ is regular, then ${\rm Crit}_{\triangle}(L_G)=\pi_0(\mathcal{N}_G)$.
\end{theorem}

\subsection{Lattices of Integral Cuts and Integral Flows}

Next, we recall the \emph{lattice of integral cuts} that will play a key role in this story and the \emph{lattice of integral flows}.


\begin{definition}{\rm {\bf (Lattices of Integral Cuts and Integral Flows)}}\cite{BacLarNag97}
The lattice of integral cuts $N^1_G$ is the image of the exterior differential $\delta$ with $A=\mathbb{Z}$, i.e. the image of $\delta:C^{0}(G,\mathbb{Z}) \rightarrow C^1(G,\mathbb{Z})$. The lattice of integral flows $\Lambda^1_G$ is the kernel of  $\delta^*: C^1(G,\mathbb{Z}) \rightarrow C^0(G,\mathbb{Z})$.
\end{definition}

The following proposition records basic properties of the two lattices (\cite[Propositions 1 and 2]{BacLarNag97}).  

\begin{proposition}
The lattices $N^1_G$ and $\Lambda^1_G$  are both sublattices of $C^1(G,\mathbb{Z})$ of rank $n-1$ and $g$, respectively.  The volumes of $N^1_G$ and $\Lambda^1_G$ are both equal to  $\sqrt{c_G}$ (recall that $c_G$ is the number of spanning trees). The two lattices are orthogonal complements of each other with respect to the inner product $\langle,\rangle_e$, i.e. $N^1_G=(\Lambda^1_G)^{\perp} \cap C^1(G,\mathbb{Z})$ and $\Lambda^1_G=(N^1_G)^{\perp} \cap C^1(G,\mathbb{Z})$. 
\end{proposition}


\subsection{ The Energy Quadratic Form}
Recall that $A$ is a commutative ring with unity. Let ${\rm Div}(G,A)$ be the free $A$-module generated by the symbols $(v)$ over all vertices $v \in V(G)$. A typical element in ${\rm Div}(G,A)$ is $\sum_{v \in V(G)} \alpha_v(v)$ where $\alpha_v \in A$ and is called an \emph{$A$-divisor}. The \emph{degree} homomorphism ${\rm deg}:{\rm Div}(G,A) \rightarrow A$ takes $\sum_{v \in V(G)} \alpha_v(v)$ to $\sum_{v \in V(G)} \alpha_v$.
For an element $d \in A$, define the subgroup ${\rm Div}^d(G,A)$ of ${\rm Div}(G,A)$ given by 
$\{ \sum_{v \in V(G)} \alpha_v (v) |~ \sum_{v \in V(G)}\alpha_v=d\}$. The free modules ${\rm Div}(G,A)$ and $C^0(G,A)$ are dual to each other and via the isomorphism given by their standard bases, we regard $\delta^{*}$ as a map from $C^1(G,A)$ to  ${\rm Div}(G,A)$. Furthermore, the image of $\delta^{*}$ is contained in ${\rm Div}^0(G,A)$. In the case when $A=\mathbb{R}$, the map $\delta^{*}$ when restricted to the cut space of $G$, namely $N^1_G \otimes_{\mathbb{Z}} \mathbb{R}$,  induces an isomorphism 
$\delta^{*}_{\rm res}:N^1_G \otimes_{\mathbb{Z}} \mathbb{R} \rightarrow  {\rm Div}^{0}(G,\mathbb{R})$.  

\begin{definition}{\rm {\bf (Energy Quadratic Form)}}\cite[Section 3.3]{BakSho13}
The energy quadratic form $\mathcal{E}_G$ on ${\rm Div}^{0}(G,\mathbb{R})$ is the pullback of the squared Euclidean norm $||.||_e^2$ on $C^{1}(G,\mathbb{R})$ along the map $(\delta^{*}_{\rm res})^{-1}: {\rm Div}^{0}(G,\mathbb{R}) \rightarrow N^1_G \otimes_{\mathbb{Z}} \mathbb{R}$.  
\end{definition}

Hence, $\mathcal{E}_G(D)=||(\delta^{*}_{\rm res})^{-1}(D)||_e^2$ for all $D \in {\rm Div}^{0}(G,\mathbb{R})$. Baker and Shokrieh \cite{BakSho13} formulate the energy quadratic form (actually, the energy pairing) in terms of inverse Laplacians.
They construct the energy quadratic form in terms of the Laplacian matrix. In terms of the Laplacian operator, their construction is as follows. The image of the Laplacian operator $\Delta: {\rm Div}(G,\mathbb{R}) \rightarrow {\rm Div}(G,\mathbb{R})$ is precisely ${\rm Div}^0(G,\mathbb{R})$. Furthermore, $\Delta$ induces an isomorphism $\Delta_{\rm res}$ when restricted to ${\rm Div}^0(G,\mathbb{R})$. Hence, $\Delta_{\rm res}^{-1}: {\rm Div}^0(G,\mathbb{R}) \rightarrow {\rm Div}^0(G,\mathbb{R})$ is well-defined. 
The energy quadratic form  $\mathcal{E}_G$ is defined as $\mathcal{E}_G(D)=\langle D, \Delta_{\rm res}^{-1}(D) \rangle_v$. The equivalence of these two definitions follows from the adjointness of $\delta$ and $\delta^{*}$ along with the definition of the Laplacian operator. 



{\bf Lattice-Metric Pairs:} Typically, a (Euclidean) lattice is defined as a discrete subgroup of Euclidean space. Implicit in this definition is a positive definite quadratic form: the squared Euclidean norm.  For the approach to Brill-Noether theory taken in this article, 
\emph{the Laplacian-energy pair}, i.e. the Laplacian lattice with respect to the energy quadratic form is central. Hence, in order to spell out the lattice and the underlying quadratic form explicitly, we formulate the notion of \emph{lattice-metric pairs}. 

Let $V_1,V_2$ be finite dimensional real vector spaces and let $\Lambda_1, \Lambda_2$ be lattices in $V_1,V_2$ respectively. Furthermore, let $Q_1,Q_2$ be positive definite quadratic forms on $V_1,V_2$ respectively. 
The pair $(V_i,Q_i)$ is usually referred as a \emph{quadratic space} in literature \cite[Chapter 2]{Cas08}. The quadratic spaces $(V_1,Q_1)$ and $(V_2,Q_2)$ are said to be \emph{isometric} if there is an isomorphism  $\sigma: V_1 \rightarrow V_2$ such that $Q_1({\bf x})=Q_2(\sigma({\bf x}))$ for all ${\bf x}\in V_1$. The pair $(\Lambda_i,Q_i)$ is called a \emph{lattice-metric pair}, in short L-M pair. 

\begin{definition}{\rm {\bf (Isomorphism of L-M Pairs)}}
Lattice-metric pairs $(\Lambda_1,Q_1)$ and $(\Lambda_2,Q_2)$ are said to be isomorphic if there is an isomorphism between $V_1$ and $V_2$ that induces a group isomorphism between $\Lambda_1$ and  $\Lambda_2$, and an isometry between the quadratic spaces $(V_1,Q_1)$ and $(V_2,Q_2)$.
\end{definition}

The \emph{sh-radius} (``shallow hole radius") of an L-M  pair $(\Lambda,Q)$ is the minimum norm $\sqrt{Q(c)}$ over all holes $c$ of $(\Lambda,Q)$, i.e. over all local maxima of $h_{Q,\Lambda}$.

Isomorphic lattice-metric pairs have congruent Voronoi cells. Hence, they have the same metric invariants such as covering radius and packing radius. In particular, the equality of the sh-radius will be important for this article and hence, we record it.

\begin{proposition} \label{sheq_prop}
Isomorphisms of lattice-metric pairs preserve the set of holes and also the set of holes supported at the origin. In particular, 
isomorphic lattice-metric pairs have the same sh-radius. 
\end{proposition}

The following isomorphism of L-M pairs will be important  to this approach. It will be used to translate problems on the Laplacian lattice $L_G$ with $\sqrt{\mathcal{E}_G}$ as the underlying metric to corresponding ones on $N^1_G$ with respect to the Euclidean metric.

\begin{proposition}\label{lapcutiso_prop}
The lattice-metric pairs $(L_G,\mathcal{E}_G)$ and $(N^1_G,||.||_e^2)$ are isomorphic via $\delta^{*}_{\rm res}: N^1_G \otimes_{\mathbb{Z}} \mathbb{R} \rightarrow  {\rm Div}^{0}(G,\mathbb{R})$.
\end{proposition}

This fact has been noted in \cite[Remark 5.13]{MohSho16}. 
We nevertheless include a proof for ease of reference. 

\begin{proof}
Since $G$ is connected, $\delta^{*}_{\rm res}$ is an isomorphism between $N^1_G \otimes_{\mathbb{Z}} \mathbb{R}$ and ${\rm Div}^{0}(G,\mathbb{R})$. Furthermore, $\delta^{*}_{\rm res}$ restricted to $N^1_G$ is an isomorphism between $N^1_G$  and $L_G$.  Finally, the property that  $\delta^{*}_{\rm res}$ preserves the underlying quadratic forms follows from the definition of $\mathcal{E}_G$. 
\end{proof}

\subsection{Riemann-Roch}

Classically, the Riemann-Roch theorem for a smooth algebraic curve $X$ asserts that the difference between the dimensions of the space of global sections of a line bundle on $X$ and its dual is  ``simple'' (essentially topological), depending only on the genus of the curve and the degree of the line bundle.  Let $K_X$ be the canonical bundle on $X$ and $g_X$ be the genus of $X$.  For any line bundle $L$ on $X$
\begin{center}
$h^0(L)-h^0(K_X \otimes L^{-1})={\rm deg}(L)-(g_X-1)$.
\end{center}
In the language of divisors, it states that for any divisor $D$ on $X$

\begin{center}
$r(D)-r([K_X]-D)={\rm deg}(D)-(g_X-1)$.
\end{center}
where $[K_X]$ is an element in the canonical divisor class. Baker and Norine in 2007 \cite{BakNor07} discovered an analogue of the Riemann-Roch theorem for graphs. \emph{The canonical divisor} $K_G$ (or simply $K$) of a graph $G$ is defined as $\sum_{v \in V(G)}({\rm val}(v)-2)(v)$. They formulated a notion of rank $r(.)$ of a divisor and showed  for any divisor $D$ on $G$,

\begin{center}
$r(D)-r(K_G-D)={\rm deg}(D)-(g-1)$.
\end{center}

Baker and Norine's original definition of rank \cite[Page 771]{BakNor07} is in terms of linear systems of divisors. We shall not state it here since we do not work with it directly. Instead, we use their degree-plus formula for rank, stated in the following. Two key components of their proof are the following:

\begin{itemize}
\item A combinatorial interpretation for the set $\mathcal{N}_G$. In particular,  every divisor class in $\mathcal{N}_G$ contains a divisor $D_{\mathcal{O}}:=\sum_{v \in V(G)}({\rm indeg}_{\mathcal{O}}(v)-1)(v)$ where $\mathcal{O}$ is an \emph{acyclic orientation} on $G$ and ${\rm indeg}_{\mathcal{O}}(v)$ is the indegree of $v$ with respect to $\mathcal{O}$. For any orientation $\mathcal{O}$,  the canonical divisor $K_G=D_{\mathcal{O}}+D_{\bar{\mathcal{O}}}$ where $\bar{\mathcal{O}}$ is the orientation opposite to $\mathcal{O}$.

\item The following degree-plus formula for rank:
\begin{equation}\label{rank_eqn} r(D)={\rm min}_{\nu \in \mathcal{N}_G} {\rm deg}^{+}(D-\nu)-1 \end{equation}
where ${\rm deg}^{+}(\tilde{D})=\sum_{v: \tilde{D}(v)>0}\tilde{D}(v)$. 
\end{itemize}

 Formula (\ref{rank_eqn}) for rank extends to $\mathbb{R}$-divisors and the Riemann-Roch formula also holds in this more general setting \cite[Theorem A.6]{Man-crc22}. 

 If ${\rm deg}(\tilde{D})=0$, then ${\rm deg}^{+}(\tilde{D})=\dfrac{\sum_{v \in V(G)}|\tilde{D}(v)|}{2}$. This observation along with Theorem \ref{critnonspe_theo} and that if $G$ is regular, then $\pi_0(\frac{K_G}{2})$ is the origin gives the following formula for the rank of the half-canonical divisor on a regular graph. 

\begin{proposition}\label{ell1rank_prop}
If $G$ is regular, then  $r(K_G/2)={\rm min}_{\nu \in \mathcal{N}_G} {\rm deg}^{+}(K_G/2-\nu)-1=\dfrac{{\rm min}_{c \in {\rm Crit}_{\triangle}(L_G)} ||c||_1}{2}-1$.
\end{proposition}

\subsection{Expanders, Ramanujans and Random Graphs}\label{expramran_subsect}


Let $\lambda_{\rm min}(G)$ and $\lambda_{\rm max}(G)$ be the minimum and maximum non-zero eigenvalues of the Laplacian matrix of $G$, respectively.  Note that $\lambda_{\rm min}(G)$ is also called the \emph{spectral gap}.


\begin{definition}{\rm {\bf (Spectral Expanders)}}
Fix an integer $k \geq 3$. A family  $\mathcal{F}=\{G_n\}_{n \in I}$ of $k$-regular graphs is called \emph{spectral expander} if there is a positive real $\mu$ such that the spectral gap $\lambda_{\rm min}(G_n) \geq \mu$ for every $n$.
\end{definition}

We refer to \cite{HooLinWig06} for more on this topic. Note that the definition of expander family \cite[Definition 2.2]{HooLinWig06} seems different: there the edge expansion ratio is to be bounded away from zero (instead of the spectral gap).  However, the two definitions are equivalent by the Cheeger inequalities \cite[Theorem 2.4]{HooLinWig06}.   Chief examples of expander graphs are \emph{Ramanujan graphs} \cite[Definition 1.1]{LubPhiSar88} (infinite families thereof). 
Before defining them, note that the vertex adjacency matrix of a $k$-regular graph ($k \geq 2$) has at least two distinct absolute eigenvalues with $k$ as the maximum.

\begin{definition}{\rm {\bf (Ramanujan Graph)}}
A $k$-regular graph is called a \emph{Ramanujan graph} if the minimum absolute second largest eigenvalue of its vertex adjacency matrix is at most $2 \cdot \sqrt{k-1}$.
\end{definition}

By definition, the spectral gap $\lambda_{\rm min}$ of a Ramanujan graph satisfies $\lambda_{\rm min} \geq k-2 \cdot \sqrt{k-1}$.  Hence, for any fixed $k \geq 3$, any infinite family of $k$-regular Ramanujan graphs is spectral expander.
In the following, we give examples of expander graphs. 

\begin{itemize}
\item The Margulis expander graph \cite[Section 8]{HooLinWig06} is a family of $8$-regular spectral expander graphs. 

  \item Let $p,q$ be odd distinct primes both congruent to one modulo four.  The Lubotzky-Phillips-Sarnak (LPS) Ramanujan graph $X^{p,q}$ is a $(p+1)$-regular graph on $\dfrac{q(q^2-1)}{2}$ or $q(q^2-1)$ vertices based on whether $\legendre[q]{p}$ is one or minus one, respectively. 
  
    \item Marcus, Spielman and Srivastava \cite{MarSpiSri15} construct families of $k$-regular graphs for any fixed integer $k \geq 3$ where each member is a bipartite, Ramanujan graph. 

\end{itemize}

Most of this article requires a slightly weaker notion than Ramanujan, namely \emph{almost-Ramanujan}.  Let $\mathcal{F}=\{G_n\}_{n \in I}$ be a family of regular graphs and let $k(n) \geq 3$ be the valence of $G_n$.

\begin{definition}{\rm {\bf(Almost-Ramanujan)}}
The family $\mathcal{F}$ is said to be \emph{almost-Ramanujan} if for every $\epsilon>0$, there is an $n_0  \in I$ such that  $\lambda_{\rm min}(G_n) \geq k(n)-2 \cdot \sqrt{k(n)-1}-\epsilon$  for every $n \in I$ that is at least $n_0$. 
\end{definition}
Note that an almost-Ramanujan family is spectral expander.  Also, any family of Ramanujan graphs is almost-Ramanujan.  Another main source of almost-Ramanujan graphs are \emph{random graphs}. More precisely, graphs drawn according to a suitable \emph{random model}.  Following Friedman \cite{Fri08}, we consider two 
random models, namely type $\mathcal{U}$ and type $\mathcal{I}$. 

A random model is a sequence  $\{\mathfrak{D}_{n,k}\}_{n \in \mathbb{N},k \in \mathbb{N}_{\geq 3}}$ where each  $\mathfrak{D}_{n,k}$ is a probability distribution on the set of $k$-regular graphs on $F(n)$ vertices for some function $F$. Type $\mathcal{U}$ consists of uniform distributions $\mathcal{U}_{n,k}$ on the set of simple, $k$-regular, connected graphs on $n$ vertices.  Note that  type $\mathcal{U}$ is supported on simple graphs while multigraphs are also included in Brill-Noether theory. To address this issue, we consider type $\mathcal{I}$. Let $\mathcal{I}_{n,k}$ be the probability distribution induced on $k$-regular graphs on $2n$ vertices formed by a union of $k$ perfect matchings on $\{1,\dots,2n\}$ drawn independently and uniformly at random. Type $\mathcal{I}$ is $\{\mathcal{I}_{n,k}\}_{n,k}$. Note that this model is supported on an even number of vertices and includes disconnected graphs.  Fix an integer $k \geq 3$,  a family $\mathcal{F}=\{G_n\}_{n \in \mathbb{N}}$ of $k$-regular graphs is said to be of 
\emph{Type $\mathcal{U}$ (or Type $\mathcal{I}$, resp.)} if each $G_n$ is drawn according to either $\mathcal{U}_{n,k}$ (or $\mathcal{I}_{n,k}$, resp.).
Friedman showed that a random family of either Type $\mathcal{U}$ or Type $\mathcal{I}$ is almost-Ramanujan \cite[Theorem 1.3 and Page 5, resp.]{Fri08}. More precisely,

\begin{theorem}\label{ranalmram_theo}
Fix an integer $k \geq 3$.  Let $\mathcal{F}$ be a family of $k$-regular random graphs of either Type $\mathcal{U}$ or Type $\mathcal{I}$.  The family  $\mathcal{F}$ is almost-Ramanujan with high probability, i.e. for every $\epsilon>0$, the probability that $\lambda_{\rm min}(G_n) \geq k-2 \cdot \sqrt{k-1}-\epsilon$ tends to one as $n \rightarrow \infty$.
\end{theorem}

The following proposition on the growth of certain quantities involving the spectrum will come in handy.

\begin{lemma}\label{almramgro_lem}
Let $\mathcal{F}$ be a family of almost-Ramanujan graphs (of either fixed or unbounded valence) and let $g(n)$ be the genus of $G_n$. The following asymptotics hold:
\begin{enumerate}
\item \label{prod_it}  The product $n \cdot \lambda_{\rm min}(G_n) \in \Omega(g(n))$.
\item \label{diff_it}  If  $\mathcal{F}$ has either unbounded or constant valence $k \geq 5$, then $\dfrac{\sqrt{n \cdot \lambda_{\rm min}(G_n)}}{4}-\dfrac{\sqrt{n}}{2\sqrt{\lambda_{\rm max}(G_n)}} \in \Omega(\sqrt{g(n)})$.
\end{enumerate}
\end{lemma}

The proof of Lemma \ref{almramgro_lem} is straightforward and is hence, omitted. Note that the lower bound $\lambda_{\rm max}(G_n) \geq k(n)+1$ can be used for Item (\ref{diff_it}). 
 We refer to \cite{HauMckYau24} for recent progress on the proportion of Ramanujan graphs.









\section{Geometric Foundations}

As we mentioned in Section \ref{prooftech_sect}, a key geometric input to Brill-Noether existence is a Cheeger-like inequality (Lemma \ref{coveglg}) for the covering radius of ${\rm Crit}_{\triangle}(L_G)$ with respect to the energy quadratic form. This section is dedicated to it. We start by showing that the set of holes of the Laplacian-energy pair $(L_G,\mathcal{E}_G)$ of a regular graph is precisely ${\rm Crit}_{\triangle}(L_G)$.

\subsection{Holes of the Laplacian-Energy Pair}\label{lapenehol_subsect}

Recall the divisor $D_{\mathcal{O}}=\sum_{v \in V(G)}({\rm indeg}_{\mathcal{O}}(v)-1)(v)$ associated to an acyclic orientation $\mathcal{O}$.

\begin{theorem}\label{lapenehol_theo}
Let $G$ be a regular, connected graph. The set of holes of the Laplacian-energy pair is ${\rm Crit}_{\triangle}(L_G)$. In particular, the set of holes supported at the origin is $\{c_{\pi}\}_{\pi \in S_n}$ where $c_{\pi}=\pi_0(D_{\mathcal{O}})$ where $\mathcal{O}$ is the acyclic orientation induced by $\pi$.
\end{theorem}

We prove Theorem \ref{lapenehol_theo} by first obtaining an analogous description for the lattice of integral cuts of $G$.  Proposition \ref{lapcutiso_prop} then allows us to transfer this to a description of the holes of the Laplacian-energy pair.

{\bf Holes of the Lattice of Integral Cuts:} We provide an explicit description of the holes of $N^1_G$ in terms of acyclic orientations of $G$. On a related note, Mohammadi and Shokrieh \cite[ Section 6.2, Theorem 5.6]{MohSho16} describe the Delaunay decomposition of $(N^1_G,||.||_e^2)$ (and of $(L_G,\mathcal{E}_G)$). Amini and Esteves \cite[Theorem A]{AmiEst20} show a bijection between the facet poset of the Voronoi cell of $(N^1_G,||.||_e^2)$ and the  poset of acyclic orientations of ``cut subgraphs of $G$".  The following description of the holes of  $N^1_G$  is implicit in \cite[Lemma 3.42 ]{AmiEst20}.

We start by formalising the notion of orientation. An orientation on $G$ is a function $\mathcal{O}: \mathbb{E}(G) \rightarrow \{1,-1\}$ such that $\mathcal{O}(e)+\mathcal{O}(\bar{e})=0$ for every $e \in \mathbb{E}(G)$. Informally, an orientation is a choice of oriented edges such that, for each $e$, precisely one of $e$ or $\bar{e}$ is picked.  For a nonempty subset $S$ of $\mathcal{O}^{-1}(1)$, let $\eta_S:=\sum_{e \in S} (e^{*}) \in C^{1}(G,\mathbb{Z})$ where $(e^{*})$ is the indicator function at $e$. We denote $\eta_{\mathcal{O}^{-1}(1)}$ by the shorthand $\eta_{\mathcal{O}}$.
 An orientation $\mathcal{O}$ is said to be \emph{acyclic} if for every nonempty subset $S$ of $\mathcal{O}^{-1}(1)$, the $1$-cochain $\eta_S \notin \Lambda^1_G$ (recall that this is the lattice of integral flows). Intuitively, $\mathcal{O}$ does not contain any directed cycles.

  Let $\{v_1,\dots,v_n\}$ be the set of vertices $V(G)$ of $G$. 
 For a permutation $\sigma$ on $V(G)$, define an orientation $\mathcal{O}_{\sigma}$ as 
 
 \[
  \mathcal{O}_{\sigma}(e)=
 \begin{cases}
 ~~1, \text{if $e=(v_{\sigma(i)},v_{\sigma(j)})$ and $i<j$}, \\
 -1, \text{otherwise}. 
 \end{cases}
 \]
 
 A straightforward fact is that an orientation is acyclic if and only if it arises from a permutation on the vertices.  Let $\pi_N: C^1(G,\mathbb{R}) \rightarrow N^1_G \otimes_{\mathbb{Z}} \mathbb{R}$ be the Euclidean orthogonal projection map. 


\begin{proposition}\label{latcuthol_prop}
 The set of holes of $(N^1_G,||.||_e^2)$ is $\{\pi_N(\eta_{\mathcal{O}}/2)+ {\bf q}|~ \mathcal{O}$ is an acyclic orientation on $G$ and ${\bf q} \in N^1_G\}$. The set of holes of $(N^1_G,||.||_e^2)$ supported at the origin is 
 $\{\pi_N(\eta_{\mathcal{O}}/2)|~ \mathcal{O}$ is an acyclic orientation on $G \}$. 
\end{proposition}

The proof strategy is to construct a simplicial refinement of the Delaunay decomposition of $(N^1_G,||.||_e^2)$. The holes are then the circumcentres of the maximal simplices of this refinement. The holes supported at the origin are 
the circumcentres of maximal simplices with the origin as a vertex. 
\begin{proof}
Let $\{v_1,\dots,v_n\}$ be the set of vertices $V(G)$ of $G$.  For a permutation $\sigma$ on $V(G)$, consider the flag $S_1 \subseteq S_2 \subseteq \dots \subseteq S_n$ where $S_i=\{v_{\sigma(1)},\dots,v_{\sigma(i)}\}$ for $i$ from $1$ to $n$.  Let $\iota_{S_j}: V(G) \rightarrow \{0,1\}$ be the indicator function of $S_j$. Note that the $1$-cochain $\delta(\iota_{S_j}) \in N^1_G$ for all $j$. Furthermore, $\delta(\iota_{S_j})=\sum_{e \in \mathbb{E}(S_j,S_j^c)}(e^{*})$ where $S_j^c:=V(G) \setminus S_j$ and $\mathbb{E}(S_j,S_j^c)$ is the set of oriented edges across the cut $(S_j,S_j^c)$, i.e. $(u,v)$ such that $u \in S_j$ and $v \in S_j^c$. Let $\mathcal{S}_{\sigma}$ be the convex hull of $\{\delta(\iota_{S_j})\}_{j=1}^{n}$. By construction, the polytope $\mathcal{S}_{\sigma}$ is a simplex \cite[Lemma 6.4]{AmiMan10}. Furthermore, the set $\{\mathcal{S}_{\sigma}+ {\bf q}\}_{\sigma \in {S_n},{\bf q} \in N^1_G}$ is a simplicial decomposition of the cut space that is supported on $N^1_G$ \cite[Lemma 6.6]{AmiMan10}. 

In the following, we show that each $\mathcal{S}_{\sigma}$ is a Delaunay simplex of $(N^1_G,||.||_e^2)$, i.e. the circumscribing ball of $\mathcal{S}_{\sigma}$ contains no lattice point ${\bf q} \in N^1_G$ in its interior. Suppose that 
$\mathcal{O}_{\sigma}$ (or simply $\mathcal{O}$) is the acyclic orientation induced by $\sigma$. By construction, for every valid $j$, every oriented edge $e \in \mathbb{E}(S_j,S_j^c)$ is contained in $\mathcal{O}^{-1}(1)$. 
Hence, \begin{center} $||\delta(\iota_{S_j})-\pi_N(\eta_{\mathcal{O}}/2)||_e=||\pi_N(\delta(\iota_{S_j})-\eta_{\mathcal{O}}/2)||_e=||\pi_N(\eta_{\mathcal{O},S_j})||_e/2$ \end{center} where $\eta_{\mathcal{O},S_j}:=\sum_{e \in  \mathbb{E}(S_j,S_j^c)} (e^{*})- \sum_{e \in \mathcal{O}^{-1}(1) \setminus \mathbb{E}(S_j,S_j^c)}(e^{*})$ for all valid $j$. Note that for all $j$, we have $||\eta_{\mathcal{O},S_j}||_e=\sqrt{m}$ where $m$ is the number of edges of $G$.  Furthermore, every point ${\bf q} \in N^1_G$ has integer coordinates and hence, 
$||{\bf q}-\eta_{\mathcal{O}}/2||_e \geq \sqrt{m}/2$. Hence, $||\delta(\iota_{S_j})-\pi_N(\eta_{\mathcal{O}}/2)||_e$ are all equal and by the Pythagoras theorem on $\pi_N$, we deduce that $||\delta(\iota_{S_j})-\pi_N(\eta_{\mathcal{O}}/2)||_e \leq ||{\bf q}-\pi_N(\eta_{\mathcal{O}}/2)||_e$ for all valid $j$ and ${\bf q} \in N^1_G$. Hence,  $\pi_N(\eta_{\mathcal{O}}/2)$ is the circumcentre of the simplex  $\mathcal{S}_{\sigma}$ and the corresponding circumscribing ball contains no lattice point in its interior. 
We conclude that $\pi_N(\eta_{\mathcal{O}}/2)$ is a Voronoi vertex of $(N^1_G,||.||_e^2)$. By general principles \cite[Chapter II, Page 33, Para. 3]{ConSlo99}, every Voronoi vertex of a lattice is a hole.

 Since the set of holes of $N^1_G$ is $N^1_G$-periodic (with respect to translation), we deduce that $\pi_N(\eta_{\mathcal{O}}/2)+{\bf q}$ is a hole for all ${\bf q} \in N^1_G$. 
Furthermore, $\mathcal{S}_{\sigma}$ is contained in the Delaunay cell corresponding to the hole $\pi_N(\eta_{\mathcal{O}}/2)$. Hence, $\cup_{\sigma {\rm ~induces~} \mathcal{O}} \{\mathcal{S}_{\sigma}\}$ decomposes the  Delaunay cell corresponding to $\pi_N(\eta_{\mathcal{O}}/2)$. We conclude that $\{\mathcal{S}_{\sigma}+{\bf q}\}_{\sigma \in S_n,{\bf q} \in  N^1_G}$ is a simplicial refinement of the Delaunay decomposition of $(N^1_G,||.||_e^2)$.  From this, we conclude that there are no other holes of $(N^1_G,||.||_e^2)$.

In the following, we characterise the holes supported at the origin. Note that if the simplex $\mathcal{S}_{\sigma}+{\bf q}$ has the origin as its vertex, then ${\bf q}=-\delta(\iota_{S_i})$ for some $i$ from $1$ to $n$.  In this case, observe that 
$\mathcal{S}_{\sigma}+{\bf q}=\mathcal{S}_{\tilde{\sigma}}$ where $\tilde{\sigma}$ is an $i$-th cyclic shift of $\sigma$. More precisely, $\tilde{\sigma}(1)=\sigma(i+1),\tilde{\sigma}(2)=\sigma(i+2),\dots,\tilde{\sigma}(n-i)=\sigma(n),\tilde{\sigma}(n-i+1)=\sigma(1), \dots,\tilde{\sigma}(n)=i$. Hence, the subset of simplices with the origin as a vertex is $\{\mathcal{S}_{\sigma}\}_{\sigma \in {S_n}}$ and the corresponding holes are $\{\pi_N(\eta_{\mathcal{O}}/2)|~ \mathcal{O}$ is an acyclic orientation on $G \}$.
\end{proof}  


For an orientation $\mathcal{O}$, define ${\rm bal}_{\mathcal{O}}=\sum_{v \in V(G)}({\rm indeg}_{\mathcal{O}}(v)-{\rm outdeg}_{\mathcal{O}}(v))(v)$ where ${\rm indeg}$ and ${\rm outdeg}$ are the indegree and outdegree, respectively.

\begin{corollary}\label{holeslapene_cor}
 The set of holes of the Laplacian-energy pair $(L_G,\mathcal{E}_G)$ is $\{ \dfrac{{\rm bal}_{\mathcal{O}}}{2} + {\bf p}|~ \mathcal{O}$ is an acyclic orientation on $G$ and ${\bf p} \in L_G \}$. The set of holes of $(L_G,\mathcal{E}_G)$  supported at the origin is 
 $\{ \dfrac{{\rm bal}_{\mathcal{O}}}{2}|~ \mathcal{O}$ is an acyclic orientation on $G \}$.
 \end{corollary}
 
 \begin{proof}
By Propositions \ref{sheq_prop} and \ref{lapcutiso_prop}, these two sets are the images of the corresponding sets of $(N^1_G,||.||_e^2)$ under the map $\delta^{*}_{\rm res}$.  We note that $\delta^{*}_{\rm res}(\pi_N(\eta_{\mathcal{O}}))=\delta^{*}(\eta_{\mathcal{O}})={\rm bal}_{\mathcal{O}}$ and apply Proposition \ref{latcuthol_prop} to complete the proof. 
 \end{proof}

 {\bf Proof of Theorem \ref{lapenehol_theo}:}
  Follows from Corollary \ref{holeslapene_cor} and noting that if $G$ is regular, then $c_{\pi}=\dfrac{{\rm bal}_{\mathcal{O}}}{2}$ for any permutation $\pi$ that induces the acyclic orientation $\mathcal{O}$. \qed


\subsection{A Cheeger Inequality}\label{che_subsect}

In the following, we derive the Cheeger lower bound on the covering radius of ${\rm Crit}_{\triangle}(L_G)$ with respect to $\mathcal{E}_G$. This goes via lower bounds on the sh-radius of $(L_G,\mathcal{E}_G)$. We first attend to this.
By Propositions \ref{sheq_prop} and \ref{lapcutiso_prop}, the sh-radii of $(N^1_G,||.||_e^2)$ and $(L_G,\mathcal{E}_G)$ are equal.   We denote this invariant by  $\tau_G$. 

\begin{lemma}\label{taulb_lem}
The invariant $\tau_G$   satisfies the following lower bounds:
\begin{enumerate}
\item \label{cheeg_ineq} The number $\tau_G  \geq \dfrac{\sqrt{\lfloor n/2 \rfloor \cdot \lambda_{\rm min}}}{2 \sqrt{2}}$.
\item \label{crit_ineq} If $G$ is $k$-regular for an odd $k$, then $\tau_G \geq \dfrac{\sqrt{n}}{2 \sqrt{\lambda_{\rm max}}}$.

\end{enumerate}
where $\lambda_{\rm min}$ and $\lambda_{\rm max}$  are the minimum and maximum non-zero eigenvalues of the Laplacian matrix of $G$, respectively.
\end{lemma}

\begin{proof}
({\bf Lower Bound \ref{cheeg_ineq}}:) By Proposition \ref{latcuthol_prop}, the sh-radius of $N^1_G$ is equal to the minimum $||\pi_N(\eta_{\mathcal{O}}/2)||_e$ over all acyclic orientations $\mathcal{O}$ of $G$. In the following, we lower bound $||\pi_N(\eta_{\mathcal{O}}/2)||_e$
for an arbitrary acyclic orientation $\mathcal{O}$.

For a non-zero $\mathcal{C} \in N^1_G \otimes_{\mathbb{Z}} \mathbb{R}$, consider the Rayleigh quotient $R( \mathcal{C}):=\dfrac{|\langle \eta_{\mathcal{O}}/2, \mathcal{C}\rangle_e|}{||\mathcal{C}||_e}$ and 
 note that, by linear algebra,  $||\pi_N(\eta_{\mathcal{O}}/2)||_e \geq R( \mathcal{C})$. Next, we construct an appropriate element 
$\mathcal{C}$ that realises the prescribed lower bound.  Let $\sigma$ be a permutation on $V(G)$ such that $\mathcal{O}$ is the acyclic orientation $\mathcal{O}_\sigma$ induced by $\sigma$. Let $S_{\lfloor n/2 \rfloor}=\{v_{\sigma(1)},\dots,v_{ \sigma( \lfloor n/2 \rfloor)}\}$ and let $\iota_{S_{\lfloor n/2 \rfloor}}$ be its indicator function.  Consider $\langle \eta_{\mathcal{O}}/2, \delta(\iota_{S_{\lfloor n/2 \rfloor}}) \rangle_e$. By construction, $\mathcal{O}^{-1}(1)$ contains the support of $\delta(\iota_{S_{\lfloor n/2 \rfloor}})$ and hence, $\langle \eta_{\mathcal{O}}/2, \delta(\iota_{S_{\lfloor n/2 \rfloor}}) \rangle_e=\dfrac{||\delta(\iota_{S_{\lfloor n/2 \rfloor}})||_e^2}{2}=\dfrac{|\mathbb{E}(S_{\lfloor n/2 \rfloor}, S_{\lfloor n/2 \rfloor}^c)|}{2}$.  Hence, the Rayleigh quotient $R(\delta(\iota_{S_{\lfloor n/2 \rfloor}})) \geq \dfrac{\sqrt{|\mathbb{E}(S_{\lfloor n/2 \rfloor}, S_{\lfloor n/2 \rfloor}^c)|}}{2}$.  By the Cheeger inequality \cite[Theorem 2.4]{HooLinWig06},  $|\mathbb{E}(S_{\lfloor n/2 \rfloor}, S_{\lfloor n/2 \rfloor}^c)| \geq \dfrac{\lambda_{\rm min} \cdot \lfloor n/2 \rfloor}{2}$. Putting these inequalities together, $||\pi_N(\eta_{\mathcal{O}}/2)||_e \geq R(\delta(\iota_{S_{\lfloor n/2 \rfloor}})) \geq \dfrac{\sqrt{\lfloor n/2 \rfloor \cdot \lambda_{\rm min}}}{2 \sqrt{2}}$ and this completes the proof. 

({\bf Lower Bound \ref{crit_ineq}:}) By Theorem \ref{lapenehol_theo},  we have 
\begin{center} $\tau_G={\rm min}_{\pi \in S_n} \sqrt{\mathcal{E}_G(c_\pi)}= {\rm min}_{\pi \in S_n} \sqrt{\langle c_{\pi}, \Delta_{\rm res}^{-1} (c_{\pi})  \rangle_v}$. \end{center}
By the Rayleigh principle,  $\langle c_{\pi}, \Delta_{\rm res}^{-1} (c_{\pi})  \rangle_v \geq \lambda_{\rm min}(\Delta_{\rm res}^{-1}) \cdot ||c_{\pi}||_v^2$ where $\lambda_{\rm min}(\Delta_{\rm res}^{-1})$ is the minimum eigenvalue of $\Delta_{\rm res}^{-1}$. 
By the definition of $\Delta_{\rm res}^{-1}$, the eigenvalue  $\lambda_{\rm min}(\Delta_{\rm res}^{-1})=\lambda_{\rm max}^{-1}$.  Hence,  \begin{equation}\label{ray_eq}  \tau_G \geq  \dfrac{||c_{\pi}||_v}{\sqrt{\lambda_{\rm max}}} \end{equation}
holds.  Note that the $j$-th coordinate  $(c_{\pi})_j=\dfrac{({\rm indeg}_{\pi}(v_j)-{\rm outdeg}_{\pi}(v_j))}{2}$ for every $j$.  Since $k$ is odd, $|(c_{\pi})_j| \geq \frac{1}{2}$ for all $j$. Hence, $||c_{\pi}||_v^2 \geq \dfrac{n}{4}$. Combining this with Equation (\ref{ray_eq}), we conclude that $\tau_G \geq \dfrac{\sqrt{n}}{2 \sqrt{\lambda_{\rm max}}}$. \end{proof}  


\begin{remark}\rm  Some related simplifications and improvements are
\rm \begin{enumerate}
\item If $k$ is odd, then $n$ is even and Lower Bound (\ref{cheeg_ineq}) simplifies to $\dfrac{\sqrt{n \cdot \lambda_{\rm min}}}{4}$.

\item Lower Bound (\ref{cheeg_ineq}) can probably be improved while retaining its ``simplicity'' if we have an explicit orthogonal basis for the cut space. However, constructing such a basis via Gram-Schmidt orthogonalisation of the standard basis $\{\delta(\iota_v)\}_{v \in V(G)}$, where $\iota_v$ is the indicator function at $v$, seems rather involved.  We refer to \cite{Wag98}  for related ideas. 

\item The proof of Lower bound (\ref{crit_ineq}) is not valid for even regular graphs since ${\rm outdeg}_{\pi}(v_j)$ and ${\rm indeg}_{\pi}(v_j)$ can be equal. However, note that the $\ell_2$-norm of $c_{\pi}$ measures the ``imbalance'' in the acyclic orientation induced by $\pi$. Lower bound (\ref{crit_ineq}) can possibly be extended to all regular graphs via this interpretation, we refer to \cite{Bor24} for related work.  
\qed
\end{enumerate} 
\end{remark}

Using Lemma \ref{taulb_lem}, we deduce the key Cheeger inequality on the covering radius of $({\rm Crit}_{\triangle}(L_G),\mathcal{E}_G)$. 


\begin{lemma}\label{coveglg}{\rm {\bf(Cheeger Inequality)}} 
The covering radius ${\rm Cov}_{\mathcal{E}_G}({\rm Crit}_{\triangle}(L_G))$ of ${\rm Crit}_{\triangle}(L_G)$ with respect to $\mathcal{E}_G$   satisfies  \begin{center} ${\rm Cov}_{\mathcal{E}_G}({\rm Crit}_{\triangle}(L_G)) \geq \dfrac{\sqrt{\lfloor n/2 \rfloor \cdot \lambda_{\rm min}}}{2 \sqrt{2}}$. 
\end{center}
 Furthermore,  this lower bound is attained at the origin $O$, i.e. $h_{\mathcal{E}_G,{\rm Crit}_{\triangle}(L_G)}(O)$ satisfies it. 
\end{lemma}
\begin{proof}
By the definition of covering radius 
\begin{center}
${\rm Cov}_{\mathcal{E}_G}({\rm Crit}_{\triangle}(L_G))={\rm max}_{{\bf p \in H_0}} {\rm min}_{c \in {\rm Crit}_{\triangle}(L_G)}\mathcal{E}_G ({\bf p}-c) \geq  {\rm min}_{c \in {\rm Crit}_{\triangle}(L_G)}\mathcal{E}_G (-c)= {\rm min}_{c \in {\rm Crit}_{\triangle}(L_G)}\mathcal{E}_G (c)=\tau_G$.
\end{center} The remaining inequality follows from Lemma \ref{taulb_lem}. By the definition of $\tau_G$, this lower bound is also attained at the origin.
\end{proof}


\section{ABNE at Half-Canonical Degree}\label{bnhc_sect}

Following the outline in  Section \ref{prooftech_sect}, we prove Theorem \ref{ramranintro_theo} by developing the third and the fourth items in Figure \ref{prftec_fig}. The Cheeger inequality (Lemma \ref{coveglg}) will be the main input to this section. 
The following proposition is a straightforward consequence of the Rayleigh principle and that the non-zero eigenvalues of the matrix underlying $\mathcal{E}_G$ are reciprocals of those of the Laplacian.


\begin{proposition}\label{raycov_prop}
For every point ${\bf p} \in H_0(:=(1,\dots,1)^{\perp})$, 
\begin{center}
$\sqrt{\dfrac{\mathcal{E}_G(\bf p)}{\lambda_{\rm max}(G)}} \leq ||{\bf p}||_2 \leq  \sqrt{\dfrac{\mathcal{E}_G(\bf p)}{ \lambda_{\rm min}(G)}}$
\end{center} 
\end{proposition}

\begin{proposition}\label{covell1delta_prop}
The covering radii ${\rm Cov}_{\ell_1}({\rm Crit}_{\triangle}(L_G))$ and ${\rm Cov}_{\triangle+\bar{\triangle}}({\rm Crit}_{\triangle}(L_G))$  of ${\rm Crit}_{\triangle}(L_G)$ with respect to the $\ell_1$-norm and the distance function $d_{\triangle+\bar{\triangle}}$, respectively  satisfy the following lower bounds:

\begin{center}

${\rm Cov}_{\ell_1}({\rm Crit}_{\triangle}(L_G)) \geq  \dfrac{\sqrt{\lfloor n/2 \rfloor \cdot \lambda_{\rm min}}}{2 \sqrt{2 \cdot \lambda_{\rm max}}}$\\ 

${\rm Cov}_{\triangle+\bar{\triangle}}({\rm Crit}_{\triangle}(L_G)) \geq  \dfrac{\sqrt{\lfloor n/2 \rfloor \cdot \lambda_{\rm min}}}{ 4\sqrt{2 \cdot \lambda_{\rm max}} \cdot n}$.

\end{center}
Both these lower bounds are attained at the origin.
\end{proposition}

\begin{proof}
By Lemma \ref{pwis-triine_lem}, Item (\ref{pwis_it}) and by Proposition \ref{raycov_prop}, ${\rm Cov}_{\ell_1}({\rm Crit}_{\triangle}(L_G)) \geq {\rm Cov}_{\ell_2}({\rm Crit}_{\triangle}(L_G)) \geq  \dfrac{{\rm Cov}_{\mathcal{E}_G}({\rm Crit}_{\triangle}(L_G))}{\sqrt{\lambda_{\rm max}}}$.  The first inequality then follows from the Cheeger lower bound for the covering radius of $(L_G,\mathcal{E}_G)$  (Lemma \ref{coveglg}).  
The second inequality follows from a simple calculation that the $\ell_1$-unit ball intersected with $H_0$ is a dilation of $\triangle+\bar{\triangle}$ by a factor of $\frac{1}{2n}$.
By Lemma \ref{coveglg}, these lower bounds are attained at the origin. \end{proof}

Recall that for a $k$-regular graph, the canonical divisor $K_G=\sum_{v \in V(G)}(k-2)(v)$. The  \emph{half-canonical divisor} on $G$ is the $\mathbb{Q}$-divisor $\frac{K_G}{2}$.


\begin{proposition}\label{halfcanrank_prop}
Let $G$ be a regular graph. The half-canonical divisor $\frac{K_G}{2}$ satisfies:
\begin{center}
$r(\frac{K_G}{2}) \geq \dfrac{\sqrt{\lfloor n/2 \rfloor \cdot \lambda_{\rm min}}}{4 \sqrt{2 \cdot \lambda_{\rm max}}}-1$.
\end{center}
\end{proposition}
\begin{proof}
Follows from Proposition \ref{ell1rank_prop} and Proposition \ref{covell1delta_prop}.
\end{proof}




Note that the lower bound for the rank of the half-canonical divisor, that Proposition \ref{halfcanrank_prop} provides, depends on the spectral gap $\lambda_{\rm min}$ of $G$. The spectral gap can be a decaying function of $n$. For example, for the $n$-cycle  $\lambda_{\rm min} \in O(1/n)$ and $\lambda_{\rm max} \xrightarrow{n \rightarrow \infty} 4$. Hence,  in this case Proposition \ref{halfcanrank_prop} only yields a constant lower bound on $r(\frac{K_G}{2})$. However, for highly-connected graphs such as expander graphs this problem does not occur and Proposition \ref{halfcanrank_prop} yields a lower bound on $r(K_G/2)$ that confirms the Brill-Noether existence conjecture asymptotically.
 


  \subsection{Theorem \ref{ramranintro_theo} for Constant Even Valence}
  
  In the following, we prove the parts of Theorem \ref{ramranintro_theo} that involve even regular graphs of constant valence, i.e. the first part along with the second and third parts when $k$ is even.
  
  \begin{proof}[Proof of (\ref{exp_cas}) and (\ref{ramcon_cas}) with $k$ even]
  Since $k$ is even,  the half-canonical divisor is a ($\mathbb{Z}$)-divisor.  Also, $\mathcal{F}$ is a family of spectral expander graphs (recall that an almost-Ramanujan family is spectral expander)  and hence, there is a $\mu>0$ such that the spectral gap $\lambda_{\rm min}(G_n) \geq \mu$ for all $n \in I$. Note that, by the Gershgorin circle bound, $\lambda_{\rm max}(G_n) \leq 2 \cdot k$. The genus $g(n)=\frac{(k-2)\cdot n}{2}+1$ and since $k$ is fixed, $\sqrt{\lfloor n/2 \rfloor} \in \Omega(\sqrt{g(n)})$. The claim then follows from Proposition \ref{halfcanrank_prop}.
  \end{proof}

  \begin{proof}[ Proof of (\ref{ran_cas}), Theorem \ref{ramranintro_theo} with $k$ even]
   By Theorem \ref{ranalmram_theo},  $\mathcal{F}$ (whether of Type $\mathcal{U}$ or Type $\mathcal{I}$) is almost-Ramanujan (and hence, connected) with high probability.  An almost-Ramanujan family is spectral expander.
   The statement then follows from Part (\ref{exp_cas}).
  \end{proof}

 \subsection{Unbounded Valence}

The lower bound for the rank of the half-canonical divisor from Proposition \ref{halfcanrank_prop} asymptotically agrees with the Brill-Noether existence conjecture for (even) expander graphs with constant valence. However, if the valence 
grows with $n$,  then so does  $\lambda_{\rm max}$ resulting in a bound that is asymptotically strictly weaker than the prediction. To alleviate this problem, we derive another lower bound on the rank of the half-canonical divisor that is based on solutions to a graphical analogue of the \emph{Poisson equation} \cite{ChuYau00}. This eliminates the intermediate step of passing through the $\ell_2$-norm in Proposition \ref{halfcanrank_prop}. 

\begin{proposition}\label{halfcanrankdia_prop}
Let $G$ be a regular graph.  The half-canonical divisor $\frac{K_G}{2}$ satisfies:
\begin{center}
$r(\frac{K_G}{2}) \geq \dfrac{\sqrt{\lfloor n/2 \rfloor \cdot \lambda_{\rm min}}}{2 \sqrt{2 \cdot {\rm diam}(G)}}-1$
\end{center}
where ${\rm diam}(G)$ is the diameter of $G$.


\end{proposition}

{\bf Proof Sketch:} The proof of Proposition \ref{halfcanrankdia_prop} directly lower bounds the $\ell_1$-norm in terms of the energy quadratic form $\mathcal{E}_G$.  Let $B$ be the $\ell_1$-unit ball intersected with $H_0$.
This method proceeds by upper bounding the energy quadratic form on $B$.   The energy quadratic form is a convex function and $B$ is a convex polytope. Hence, attains its maximum at a vertex of $B$. Hence, it suffices to upper bound $\mathcal{E}_G$ on the vertices of $B$. The vertex set of $B$ is, up to a $1/2$-factor, the root system of type $A$, namely $\{\dfrac{e_i-e_j}{2}\}_{i \neq j}$ where $\{e_1,\dots,e_n\}$ is the standard basis of $\mathbb{R}^n$.  This leads to the maximisation problem ${\rm max}_{i,j} \mathcal{E}_G(e_i-e_j)$.
By definition, $\mathcal{E}_G(e_i-e_j)=\langle e_i-e_j, \Delta_{\rm res}^{-1}(e_i-e_j)\rangle_v$. For any valid $i,j$, there is an $f_{i,j} \in C^0(V,\mathbb{Q})$ such that $f_{i,j}=\Delta_{\rm res}^{-1}(e_i-e_j)$. Such an $f_{i,j}$ is a solution to the Poisson equation  $\Delta(f)=e_i-e_j$
and is a discrete analogue of the \emph{Green's function}. This leads to graph theoretic analogues of Green's functions, a topic taken up in \cite{ChuYau00, BakSho13}. Using this, we upper bound $\langle e_i-e_j, f_{i,j} \rangle_v$. \qed

{\bf Poisson Equation on Graphs:} Let $G$ be a graph and $D \in {\rm Div}(G,\mathbb{Z})$ (also common are ${\rm Div}(G,\mathbb{Q})$ or ${\rm Div}(G,\mathbb{R})$), the equation 

\begin{equation}\label{grapoi_eq}
\Delta(f)=D
\end{equation} 

is called the \emph{Poisson equation} on $G$ (here,  the domain of $\Delta$ is $C^0(G,\mathbb{R})$ and its codomain is ${\rm Div}(G,\mathbb{R})$). The Poisson equation has a solution only if the degree of $D$ is zero.  In this case, it has a solution in $C^0(G,\mathbb{Q})$ that is unique modulo the constant function. Consider the case when $D=(s)-(t)$ for vertices $s,~t$ of $G$. The unique solution $f$  with $f(t)=0$  to the corresponding Poisson equation is called the \emph{Green's function} $\mathfrak{g}_{s,t}$.  We refer to \cite[Section 3.2]{BakSho13} for an interpretation of $\mathfrak{g}_{s,t}$ in terms of electrical networks (where it is referred to as the \emph{$j$-function} $j_t(s,.)$).  In the following, we record its well-known key properties:

\begin{proposition}\label{gre_prop}
For any pair of vertices $s,t$ of $G$, the function $\mathfrak{g}_{s,t}$ has the following properties:
\begin{enumerate}
\item\label{gremaxmin_it} It is maximised and minimised at vertices $s$ and $t$, respectively.

\item\label{grediv_it} {\bf (Green's Theorem)}  For any real $\kappa < \mathfrak{g}_{s,t}(s)$, the sublevel set $S_{\leq \kappa}:=\{v \in V(G)|~\mathfrak{g}_{s,t}(v) \leq \kappa\} \subset V(G)$. It satisfies $(\delta(\mathfrak{g}_{s,t}))(e) >0$ for any $e \in \partial{S_{\leq \kappa}}$ and the following divergence property:

 \begin{center} $\sum_{e \in \partial{S_{\leq \kappa}}} (\delta(\mathfrak{g}_{s,t}))(e)=\sum_{v \in S_{\leq \kappa}} (\Delta(\mathfrak{g}_{s,t}))(v)=1$ \end{center}
where $\partial{S_{\leq \kappa}}$ is $\mathbb{E}(S_{\leq \kappa},S_{\leq \kappa}^c)$ (the set of oriented edges across the cut $(S_{\leq \kappa},S_{\leq \kappa}^c)$) and $(\Delta(\mathfrak{g}_{s,t}))(v)$ is the coefficient of $v$ in the divisor $\Delta(\mathfrak{g}_{s,t})$.
\item\label{greabs_it} For any edge $e \in \mathbb{E}(G)$, the absolute value $|\delta(\mathfrak{g}_{s,t}))(e)| \leq 1$.
\end{enumerate}
\end{proposition}
\begin{proof}(Sketch)  We refer to \cite[Section 3.1]{BakSho13} for the proof of Item (\ref{gremaxmin_it}). Item (\ref{grediv_it}) follows from induction on the level $\kappa$. Item (\ref{greabs_it}) follows from Item (\ref{grediv_it}) and the observation that for every oriented edge $e$ with $(\delta(\mathfrak{g}_{s,t}))(e) \neq 0$, either $e$ or $\bar{e}$ 
is in $\partial{S_{\leq \kappa}}$ for some $\kappa>0$.\end{proof}

Proposition \ref{gre_prop} yields the following upper bound on ${\rm max}_{i,j} \mathcal{E}_G(e_i-e_j)$.

\begin{corollary}\label{enedia_cor}
For any $i,j$ from $1$ to $n$, the following upper bound holds:
\begin{center}
$\mathcal{E}_G(e_i-e_j) \leq {\rm diam}(G)$.
\end{center}
\end{corollary}


\begin{proof}
By the definition of the energy quadratic form and the Green's function, 
 $\mathcal{E}_G(e_i-e_j)=\mathfrak{g}_{i,j}(i)-\mathfrak{g}_{i,j}(j)=\mathfrak{g}_{i,j}(i)$ and by Item \ref{gremaxmin_it}, Proposition \ref{gre_prop}, this is ${\rm max}_{v \in V(G)}  \mathfrak{g}_{i,j}(v)$. Suppose that 
$P=(\xi_1,\dots, \xi_l)$ is any directed path from $i$ to $j$.  The evaluation $\mathfrak{g}_{i,j}(i)=|(\delta(\mathfrak{g}_{i,j}))(\xi_1)+\dots+(\delta(\mathfrak{g}_{i,j}))(\xi_l)| \leq  |(\delta(\mathfrak{g}_{i,j}))(\xi_1)|+\dots+|(\delta(\mathfrak{g}_{i,j}))(\xi_l)|$
and by Item (\ref{greabs_it}),  Proposition \ref{gre_prop}, $|(\delta(\mathfrak{g}_{i,j}))(\xi_1)|+\dots+|(\delta(\mathfrak{g}_{i,j}))(\xi_l)| \leq l$.  Hence, the corollary. 
\end{proof}

{\bf Proof of Proposition \ref{halfcanrankdia_prop}:} By Corollary \ref{enedia_cor}, for every ${\bf p} \in H_0$, the inequality
\begin{equation}\label{ell1gre_eq}
||{\bf p}||_1 \geq \dfrac{2 \sqrt{\mathcal{E}_G(\bf p})}{\sqrt{{\rm diam}(G)}}
\end{equation}
holds.  By Proposition \ref{ell1rank_prop}, $r(\frac{K_G}{2})=\frac{{\rm min}_{c \in {\rm Crit}_{\triangle}(L_G)} ||c||_1}{2}-1$.  Hence, $r(\frac{K_G}{2}) \geq \dfrac{\sqrt{ {\rm min}_{c \in {\rm Crit}_{\triangle}(L_G)}  \mathcal{E}_G(c)}}{\sqrt{{\rm diam}(G)}}-1=\dfrac{\tau_G}{\sqrt{{\rm diam}(G)}}-1$.  Combining this with Item (\ref{cheeg_ineq}), Lemma \ref{taulb_lem} yields the claim. \qed

The diameter of almost-Ramanujan graphs is at most logarithmic in the number of vertices, more precisely $O(\log_{k(n)}n)$ \cite{Chu89}.  This fact along with Proposition \ref{halfcanrankdia_prop} gives a $\mathbb{Q}$-version of Part (\ref{ramunb_cas}), Theorem \ref{ramranintro_theo}.



\begin{proposition} With the setting of Part (\ref{ramunb_cas}), Theorem \ref{ramranintro_theo}, the family $\{\frac{K_{G_n}}{2}\}_{n \in I}$ of half-canonical $\mathbb{Q}$-divisors satisfies $r(\frac{K_{G_n}}{2}) \in \Omega(\sqrt{\dfrac{g(n)}{\log_{k(n)}n}})$. \end{proposition}
\begin{proof}
 Consider the lower bound for the rank of $\frac{K_{G_n}}{2}$ from Proposition \ref{halfcanrankdia_prop}.  Since $\mathcal{F}$  is an almost-Ramanujan family,  $\sqrt{\lfloor n/2 \rfloor \cdot \lambda_{\rm min}(G_n)} \in  \Omega(\sqrt{g(n)})$ by  Item (\ref{prod_it}), Lemma \ref{almramgro_lem}.  By \cite[Theorem 1]{Chu89}, the diameter of $G_n$ is in $O(\log_{k(n)}n)$. Hence, the claim. 
 \end{proof}

If each $k(n)$ is even, then the half-canonical divisor $\frac{K_{G_n}}{2}$ is a  $\mathbb{Z}$-divisor and we have:

  \begin{corollary}
  In the case where $k(n)$ is even for all $n$, Part (\ref{ramunb_cas}), Theorem \ref{ramranintro_theo} holds.
 \end{corollary}








 

\subsection{From $\mathbb{Q}$ to $\mathbb{Z}$: the Odd Case}\label{qtoz_sect}

In this subsection, we construct $\mathbb{Z}$-divisors promised by Theorem \ref{ramranintro_theo} in the case where the $\mathbb{Q}$-versions are already not integral.  This is precisely for odd regular graphs. The key idea is to construct $\mathbb{Z}$-divisors whose projection onto $H_0$ (via $\pi_0$)  is \emph{``close enough" to the origin} with respect to the \emph{energy metric} $\sqrt{\mathcal{E}_G(.)}$. The triangle inequality on the energy metric along with norm conversion  then yields the lower bounds on the rank. These divisors are constructed by adding suitable \emph{multiples of supports of certain vertex subsets} to the half-canonical divisor.

Let $G$ be an odd regular graph. In this case, the half-canonical divisor has half-integral coefficients and the number of vertices $n$ is even. For a subset $S$ of $V(G)$ of cardinality $n/2$, let $D_S \in {\rm Div}^0(G,\mathbb{Q})$ be defined as follows

\begin{center}
$D_S(v)=
\begin{cases}
~~1/2, \text{ if $v \in S$},\\
-1/2, \text{ otherwise}.
\end{cases}
$.\end{center}

The $\mathbb{Q}$-divisor $D_S$ has degree zero and by the Rayleigh inequality, we have

\begin{lemma}\label{enenorup_lem}
The energy norm $\sqrt{\mathcal{E}_G(D_S)}$ at $D_S$ satisfies:
\begin{center}
$\sqrt{\mathcal{E}_G(D_S)} \leq \dfrac{\sqrt{n}}{ 2 \sqrt{\lambda_{\rm max}(G)}}$.
\end{center}
\end{lemma}

This leads to the following lower bound on $h_{\mathcal{E}_G, {\rm Crit}_{\triangle}(L_G)}(D_S)$.


\begin{lemma}\label{hlbint_lem}
Let $\mathcal{F}$ be a  family of almost-Ramanujan graphs of odd valence that is either a fixed $k \geq 5$ or is unbounded. For any family $\{S^{(n)}\}_{n \in I}$ where $S^{(n)}$ is a subset of size $|V(G_n)|/2$, the asymptotic lower bound
\begin{center}
$h_{\mathcal{E}_{G_n}, {\rm Crit}_{\triangle}(L_{G_n})}(D_{S^{(n)}}) \in \Omega(\sqrt{g(n)})$
\end{center}
holds.
\end{lemma}


\begin{proof}
The triangle inequality on $h_{\mathcal{E}_{G_n}, {\rm Crit}_{\triangle}(L_{G_n})}$(Lemma \ref{pwis-triine_lem}, Item (\ref{triine_it})) yields
\begin{center}
$h_{\mathcal{E}_{G_n}, {\rm Crit}_{\triangle}(L_{G_n})}(D_{S^{(n)}}) \geq h_{\mathcal{E}_{G_n}, {\rm Crit}_{\triangle}(L_{G_n})}(O)-\sqrt{\mathcal{E}_{G_n}(D_{S^{(n)}})}$
\end{center}
where $O$ is the origin.  By the Cheeger inequality (Lemma \ref{coveglg}), $h_{\mathcal{E}_{G_n}, {\rm Crit}_{\triangle}(L_{G_n})}(O) \geq \dfrac{\sqrt{n \cdot \lambda_{\rm min}(G_n)}}{4}$. Combining this with Lemma \ref{enenorup_lem}, 
$h_{\mathcal{E}_{G_n}, {\rm Crit}_{\triangle}(L_{G_n})}(D_{S^{(n)}}) \geq \dfrac{\sqrt{n \cdot \lambda_{\rm min}(G_n)}}{4}-\dfrac{\sqrt{n}}{2 \sqrt{{\lambda_{\rm max} (G_n)}}}$.   Lemma \ref{almramgro_lem}, Item (\ref{diff_it}) then completes the proof.  \end{proof}



 


Throughout the following proof, $S^{(n)}$ is an arbitrary vertex subset of $G_n$ of size $n/2$.

{\bf Proof of the Odd Parts of Items (\ref{ramcon_cas}), (\ref{ran_cas}) and Item (\ref{ramunb_cas}) of Theorem \ref{ramranintro_theo}:}
For an odd regular $G_n$,  note that $\frac{K_{G_n}}{2}+D_{S^{(n)}}$ is a $\mathbb{Z}$-divisor of half-canonical degree. 
For the constant valence case, consider the family of divisors $\{\frac{K_{G_n}}{2}+D_{S^{(n)}}\}_{n \in I}$. Note that $\pi_0(\frac{K_{G_n}}{2}+D_{S^{(n)}})=D_{S^{(n)}}$.  Lemma  \ref{hlbint_lem} and Proposition \ref{raycov_prop} (along with  Lemma \ref{pwis-triine_lem} (Item \ref{pwis_it}) and Proposition \ref{ell1rank_prop}) yield $r(\frac{K_{G_n}}{2}+D_{S^{(n)}}) \in \Omega(\sqrt{g(n)})$. This proves Item (\ref{ramcon_cas}). Theorem \ref{ranalmram_theo} then yields Item (\ref{ran_cas}). 
For Item \ref{ramunb_cas}, consider the family of divisors $\{D_n\}_{n \in I}$ defined as follows:

\begin{center}
$D_n=
\begin{cases}
\frac{K_{G_n}}{2}, \text{ if $k(n)$ is even},\\
\frac{K_{G_n}}{2}+D_{S^{(n)}}, \text{ if $k(n)$ is odd}.
\end{cases}
$.\end{center}

By Lemma \ref{hlbint_lem} and Equation (\ref{ell1gre_eq}) (with Proposition \ref{ell1rank_prop}), $r(D_n)$ (as a function of $n$) satisfies the claimed lower bounds.
\qed

\section{Reversal Systems on Graphs}\label{revsys_sect}
 
 We present an application of Brill-Noether existence at half-canonical degree to certain \emph{dynamical systems} on graphs, namely, the \emph{cocycle reversal} system and the \emph{cycle-cocycle reversal} system.
In their basic form, these reversal systems are certain equivalence relations defined on the set of orientations of a graph. We refer to Backman's work \cite{Bac17} for a generalisation to partial orientations. 
The equivalence relations are defined in terms of certain operations called the \emph{cycle reversal} and the \emph{cocycle reversal}. Our definitions of these operations will be informal and we refer to \cite[Sections 1 and 2]{Bac17} for a formal treatment. 

 A \emph {cycle reversal} operation on an orientation reverses a directed cycle.  A \emph{cocycle reversal} operation reverses a consistently oriented directed cut. In a cycle-cocycle reversal system, graph orientations $\mathcal{O}_1$ and 
 $\mathcal{O}_2$ are \emph{equivalent} if there is a finite sequence consisting of cycle and cocycle reversals that transforms  $\mathcal{O}_1$ to $\mathcal{O}_2$. The equivalence relation in the cocycle reversal system is defined analogously allowing only cocycle operations. 
 
 Backman interprets the rank of a ($\mathbb{Z}$-)divisor in terms of certain graphs associated to reversal systems. We shall only recall this result for divisors of degree $g-1$. For a graph $G$, define an undirected graph $\mathcal{G}_{\rm cc}$ called the \emph{path reversal} graph (with respect to cycle-cocycle reversal) as follows. Its vertices are equivalence classes of orientations of $G$ (in the cycle-cocycle reversal system) and its edges are $([\mathcal{O}_1],[\mathcal{O}_2])$ whenever there are orientations $\tilde{\mathcal{O}_1} \in  [\mathcal{O}_1]$ and $\tilde{\mathcal{O}_2} \in  [\mathcal{O}_2]$ such that there is a directed path reversal that transforms $\tilde{\mathcal{O}_1}$ to $\tilde{\mathcal{O}_2}$. 
 The path reversal graph $\mathcal{G}_{\rm cy}$ with respect to the cocycle reversal system is defined analogously.   Studying these path reversal graphs is a natural approach to reversal systems.
 
Recall that the divisor $D_{\mathcal{O}}$ associated to an orientation $\mathcal{O}$ is defined as $\sum_{v \in V(G)}({\rm indeg}_{\mathcal{O}}(v)-1)(v)$. Such a divisor is called \emph{orientable}.  By \cite[Theorem 4.10]{AnBakKupSho14},  every divisor of degree $g-1$ is linearly equivalent to an orientable divisor.  Let $\mathcal{A}=\{[\mathcal{O}_a]|~  
\mathcal{O}_a \text{~is an acyclic orientation on~} G\}$.
 
 \begin{theorem}{\rm \bf{(Backman \cite[Theorem 5.8]{{Bac17}})}} \label{ranpatrev_theo}
Let $D$ be any divisor on $G$ of degree $g-1$ and let $D_{\mathcal{O}}$ be an orientable divisor linearly equivalent to it.  The rank of $D$ is one less than the distance from $[\mathcal{O}]$ to the set $\mathcal{A}$ in $\mathcal{G}_{\rm cc}$.
\end{theorem}

An analogous statement also holds for the cocycle reversal system. In the following, we transfer the lower bounds on the rank of the half-canonical divisor (and its ``twists") from Section \ref{bnhc_sect} to the diameter of path reversal graphs.

Fix a non-negative even integer $k$. Let $\mathcal{F}=\{G_n\}_{n \in I}$ be a family of $k$-regular spectral expander graphs. Let $\mathcal{G}_{{\rm cc},n}$ and $\mathcal{G}_{{\rm cy},n}$ be the path reversal graphs of $G_n$ with respect to the cycle-cocycle reversal system and the cocycle reversal system, respectively. 

\begin{theorem}\label{prdilb_theo}
The diameters of $\mathcal{G}_{{\rm cc},n}$ and $\mathcal{G}_{{\rm cy},n}$ are both in $\Omega(\sqrt{n})$.
\end{theorem}

Theorem \ref{prdilb_theo} is a direct consequence of Theorem \ref{ranpatrev_theo} and Item (\ref{exp_cas}), Theorem \ref{ramranintro_theo}. In particular, the divisor associated to any Eulerian orientation on $G$ is half-canonical and in both the reversal systems, the distance of the Eulerian orientation class of $G_n$ from the set $\mathcal{A}(G_n)$  is in $\Omega(\sqrt{n})$.  Analogous results also hold for the random models $\mathcal{U}_{n,k}$ and $\mathcal{I}_{n,k}$ when $k \geq 5$. 

\begin{theorem}\label{prdilbran_theo}
Fix an integer $k \geq 5$.  There is a constant $C$ (independent of $n$) such that a random regular graph $G$ drawn according to $\mathcal{U}_{n,k}$ or $\mathcal{I}_{n,k}$ satisfies the following property with high probability:
\begin{center}
  The diameters of the path reversal graphs with respect to both the cycle-cocycle reversal system and the cocycle reversal system is at least $C \cdot \sqrt{n}$.
 \end{center}
\end{theorem}

The corresponding statements for high degree Ramanujan graphs are as follows.

\begin{theorem}\label{ramprdilb_theo}
 Let $\mathcal{F}=\{G_n\}_{n \in I}$ be a family of almost-Ramanujan graphs and let $k(n) \geq 3$ be the valence of $G_n$. The diameters of $\mathcal{G}_{{\rm cc},n}$ and $\mathcal{G}_{{\rm cy},n}$ are both in $\Omega(\sqrt{n})$ if $k(n)=k \geq 5$ is a constant and in $\Omega(\sqrt{\dfrac{g(n)}{\log_{k(n)}n}})$ if $k(n)$ is unbounded.
\end{theorem}



Theorems  \ref{prdilbran_theo} and \ref{ramprdilb_theo} are direct consequences of Theorem \ref{ranpatrev_theo} combined with Items (\ref{ramcon_cas}),(\ref{ran_cas}) and (\ref{ramunb_cas}) of Theorem \ref{ramranintro_theo}.


%




\section{Going Beyond Asymptotics and Half-Canonical Degrees}\label{bey_sect}

Both these directions seem to need other quadratic forms on the Laplacian lattice with the same set ${\rm Crit}_{\triangle}(L_G)$ as its set of holes. For this, the following weighted generalisation of the energy quadratic form may be helpful.  Let $w: \mathbb{E}(G) \rightarrow \mathbb{R}_{>0}$
be a function such that  $w(e)=w(\bar{e})$ for every oriented edge $e$. The $w$-weighted squared Euclidean norm on $C^{1}(G,\mathbb{R})$ is given by \begin{center} $||\sum_{e \in \mathbb{E}(G)} \alpha_e (\iota_e)||_{e,w}^2:=\sum_{e \in \mathbb{E}(G)} w_e \alpha_e^2$ \end{center}
where $w_e:=w(e)$ and $\iota_e$ is the indicator function at $e$. 

\begin{definition}{\rm {\bf (Weighted Energy Quadratic Form)}}
The $w$-weighted energy quadratic form $\mathcal{E}_{G,w}$ on $L_G$ is the pullback of the $w$-weighted squared Euclidean norm along the map $(\delta^{*}_{\rm res})^{-1}$. 
\end{definition}

Note that the weight function $w$ being constant one corresponds to the energy quadratic form. Theorem \ref{lapenehol_theo} can be generalised to the \emph{weighted-Laplacian-energy pair} $(\mathcal{E}_{G,w},L_G)$.

\begin{theorem}
Let $G$ be a regular, connected graph.  The set of holes of the weighted-Laplacian-energy pair  $(\mathcal{E}_{G,w},L_G)$ is ${\rm Crit}_{\triangle}(L_G)$.\end{theorem}

The proof follows the same lines as that of Theorem \ref{lapenehol_theo}. For every graph $G$ and a pair $(d_0, r_0)$ satisfying  $\rho(g,r_0,d_0)\geq 0$, whether there is a weight function $w$ such that $\mathcal{E}_{G,w}$, via a Cheeger inequality, yields Brill-Noether existence at $(d_0, r_0)$  is  a topic for future work. 




\footnotesize
\noindent {\bf Author's address:}

\smallskip

\noindent Department of Mathematics,\\
Indian Institute of Technology Bombay,\\
Powai, Mumbai,
India 400076.\\

\noindent {\bf Email id:} madhu@math.iitb.ac.in, madhusudan73@gmail.com.\\

\end{document}